\documentclass[12pt]{amsart}
\usepackage{amssymb,pslatex}

\pdfoutput=1 
 
\newtheorem{thm}{Theorem}[section]

\newtheorem{lem}[thm]{Lemma}
\newtheorem{prop}[thm]{Proposition}
\theoremstyle{definition}
\newtheorem{dft}[thm]{Definition}
\theoremstyle{remark}

\newtheorem{ntn}[thm]{Notation}
\newtheorem{num}[thm]{}

\def\syl{{\rm Syl}}

\def\fk{{\mathrm{k}}}
\def\ind{{\rm Ind}}

\def\Bb{{\mathcal B}}
\def\Cc{{\mathcal C}}
\def\Dd{{\mathcal D}}
\def\Ss{{\mathcal S}}

\def\wBb{\widehat{\mathcal B}}
\def\wCc{\widehat{\mathcal C}}
\def\wSs{\widehat{\mathcal S}}
\def\tCc{\widetilde{\mathcal C}}
\def\tSs{\widetilde{\mathcal S}}
\def\tH{\widetilde{H}}
\def\tL{\widetilde{L}}

\begin{document}

\title[Vertices of Lefschetz modules]{On the vertices of indecomposable summands of certain Lefschetz modules}

\author{John Maginnis}
\author{Silvia Onofrei}

\address{Department of Mathematics Department, Kansas State University, 137 Cardwell Hall, Manhattan, Kansas 66506}
\email{maginnis@math.ksu.edu}

\address{Department of Mathematics, The Ohio State University, 100 Mathematics Tower, 231 West $18^{\rm th}$ Avenue, Columbus, Ohio 43210}
\email{onofrei@math.ohio-state.edu}

\date{2 September 2011}

\thanks{The research of the second author was partially supported by the National Security Agency under grant number H98230-10-1-0174.}

\maketitle

\begin{abstract}
We study the reduced Lefschetz module of the complex of $p$-radical and $p$-centric subgroups. We assume that the underlying group $G$ has parabolic characteristic $p$ and the centralizer of a certain noncentral $p$-element has a component with central quotient $H$ a finite group of Lie type in characteristic $p$. A non-projective indecomposable summand of the associated Lefschetz module lies in a non-principal block of $G$ and it is a Green correspondent of an inflated, extended Steinberg module for a Lie subgroup of $H$. The vertex of this summand is the defect group of the block in which it lies.
\end{abstract}

\section{Introduction}

The best known example of a reduced Lefschetz module is the Steinberg module for a finite simple group of Lie type, with coefficients having the defining characteristic. This modular representation is both irreducible and projective. The Steinberg module equals the top homology module of the associated Tits building, but can also be described using an alternating sum of the modules of the chain complex of the building (and it is this type of virtual module which we will refer to as a Lefschetz module).

Another important example is the Brown complex of an arbitrary finite group, with simplices the chains of inclusions of nontrivial $p$-subgroups for a fixed prime $p$. The reduced Lefschetz module of the Brown complex, with coefficients in a field of characteristic $p$, although not necessarily indecomposable is always projective both on the level of characters in the Grothendieck ring \cite[Corollary 4.3]{qu78} and as a virtual module in the Green ring \cite{wb87a}. More generally, Webb \cite[Theorem $A'$]{wb87a} proved that if the fixed point sets of elements of order $p$ are contractible, for an admissible action of a finite group on a finite simplicial complex, then the reduced Lefschetz module is a projective virtual module (he worked with $p$-adic integers as coefficients).

Ryba, Smith and Yoshiara \cite{rsy} applied Webb's Theorem to various geometries admitting a flag transitive action by a sporadic or alternating group, or such geometries other than the building for Lie type groups.  Their Table 1 summarizes information on 18 cases where the reduced Lefschetz module is projective, including descriptions of many indecomposable summands as projective covers of irreducible modules. Later work of Smith and Yoshiara \cite{sy} used a different approach, verifying equivariant homotopy equivalences between such sporadic geometries and the simplicial complex for the Quillen poset of nontrivial elementary abelian $p$-subgroups.

Other subgroup complexes have been constructed with applications in representation theory and to homology decompositions of classifying spaces; see for example \cite{dwh}. The collection of nontrivial $p$-centric and $p$-radical subgroups $\Dd_p(G)$ is often regarded as the best candidate for a $p$-local geometry and a generalization of the notion of building to general finite groups.

A comprehensive study of $2$-local sporadic geometries (in which all the vertex stabilizers are $2$-local subgroups) with applications to mod-$2$ cohomology can be found in Benson and Smith \cite{bs04}. In many cases, notably when the group $G$ has parabolic characteristic $2$ (see Section $2$ for an explanation of this term), the $2$-local geometries studied by Benson and Smith are $G$-homotopy equivalent with the complex of $2$-centric and $2$-radical subgroups.

One of the first to study nonprojective reduced Lefschetz modules,  Sawabe \cite[Prop. 5]{sa06} determined an upper bound for the orders of vertices of their indecomposable summands. Further results regarding the Lefschetz characters associated to several $2$-local sporadic geometries were obtained by Grizzard \cite{gri}. He studied the distribution into blocks of the indecomposable summands, for several sporadic groups whose $2$-local geometries have nonprojective Lefschetz characters.

Since the properties of the reduced Lefschetz modules are strongly influenced by the group structure and, in particular by the centralizers of $p$-elements, it seems natural to study these properties under appropriately chosen conditions imposed on the $p$-local structure of $G$. A $p$-central element is a nontrivial element of $G$ that lies in the center of a Sylow $p$-subgroup of $G$. In \cite{mgo3} we introduced distinguished collections of $p$-subgroups, those $p$-subgroups of some standard collection that contain $p$-central elements in their centers. They have similar properties as their standard counterparts. If $G$ has parabolic characteristic $p$, then the distinguished collection of $p$-radical subgroups $\wBb_p(G)$ is equal to the collection of $p$-centric and $p$-radical subgroups $\Dd_p(G)$.

In the present paper we study the reduced Lefschetz module $\tL=\tL_G(\wBb_p(G); \fk) = \tL_G(\Dd_p(G); \fk)$ when $G$ has parabolic characteristic $p$. Our approach is to use a theorem of Robinson, quoted in Section $5$ of this paper as Theorem \ref{fixedrob}, which relates the indecomposable summands with vertex $V$ in $\tL_G(\Delta; \fk)$ to the indecomposable summands with vertex $V$ in $\tL_{N_G(V)}(\Delta^V; \fk)$, for any simplicial complex $\Delta$.

We assume for most of this paper that $G$ has parabolic characteristic $p$. We first determine the fixed point sets of $p$-subgroups $T$ acting on $\Delta=\wBb_p(G)$. If the largest normal $p$-subgroup $O_p(T \cdot C_G(T))$ of $T$ times its centralizer contains $p$-central elements then the fixed point set $\Delta^T$ is contractible; this is Theorem \ref{ocintilde}. If $O_p(T \cdot C_G(T))$ does not contain $p$-central elements and if $T \cdot C_G(T)$ has a single component $H$ that also has parabolic characteristic $p$ then the fixed point set is equivariantly homotopy equivalent to the complex of distinguished $p$-radical subgroups in $H$. This result is proved in Theorem \ref{noncenthm}.

Next we study the vertices of the indecomposable summands of $\widetilde{L}$. We focus on the case when the centralizer of certain noncentral $p$-elements has a component $H$ which is a finite group of Lie type in characteristic $p$. We find in Theorem \ref{defT}, that a nonprojective indecomposable summand of $\tL$ is the Green correspondent of an inflated, extended Steinberg module which corresponds to $H$. This summand lies in a non-principal block of the group ring $\fk G$, and has vertex equal to the defect group of the block in which it lies.

We discuss applications to the sporadic simple groups of parabolic characteristic $p$ for $p=2$ and $3$. Our results verifies for these groups a conjecture of Grizzard \cite{gri} and of Benson and Smith \cite[p.164]{bs04} who noticed: {\it``...for all $15$ sporadic groups $G$ ... for which the $2$-local geometry $\Delta$ is not homotopy equivalent to $\Ss_2(G)$, it seems that the reduced Lefschetz module involves an indecomposable in a suitable non-principal $2$-block of $G$ of small positive defect."} Specifically, we give an explanation for $12$ of the sporadic simple groups that have parabolic characteristic $2$.

\medskip
We use the standard notation for finite groups as in the Atlas \cite{atlas}. If $p$ is a prime, then $p^n$ denotes an elementary abelian $p$-group of rank $n$ and $p$ a cyclic group of order $p$. Also $H.K$ denotes an extension of the group $H$ by a group $K$ (it does make no statement if the extension splits or not) and $H:K$ denotes a split extension, while $H ^.K$ indicates a nonsplit extension. Finally $H \cdot K$ denotes the product of two subgroups of a group $G$, with at least one normal in $G$. Regarding the conjugacy classes of elements of order $p$ we follow the Atlas \cite{atlas} notation. The simplified notation of the form $5A^2$ stands for an elementary abelian group $5^2$ whose $24$ nontrivial elements are all of type $5A$. The notation of the form $2A_2B_1$ stands for a group of order $4$ which contains two elements from class $2A$ and one element from class $2B$.

\medskip
The structure of our paper is as follows. In Section $2$ we review properties of groups of local characteristic $p$ and parabolic characteristic $p$, as well as standard facts on collections of $p$-subgroups. We also recollect a few useful results on distinguished collection of $p$-subgroups. In Section $3$ we determine the nature of the fixed point sets under the action of $p$-subgroups. Section $4$ is dedicated to standard facts from modular representation theory. In Section $5$ we prove a theorem that characterizes the Lefschetz module of the collection of $p$-centric and $p$-radical subgroups. Applications to the sporadic simple groups are discussed in Section $6$.

\subsubsection*{Acknowledgements} We would like to thank Ron Solomon for providing us with an elegant argument for the existence of central elements in the $3$-core of the centralizer of an element of type $3D$ in the groups $Fi_{22}$ and $Fi_{23}$ respectively.

\section{Collections of p-subgroups}

In this section we introduce the necessary notation, review the basic terminology regarding subgroup complexes and list a few results from the literature which will be used later in the proofs. We assume that $G$ is a finite group and $p$ is a prime dividing its order.

\subsection{Groups of characteristic p and local characteristic p}

The finite groups we are working with in the present paper have properties known as characteristic $p$, local characteristic $p$ and parabolic characteristic $p$. For completeness of our presentation, in this section we define these notions and list some of their features. Most of the results mentioned here are well known and can be found elsewhere in the literature.

If $G$ is a finite group, we denote by $O_p(G)$ the largest normal $p$-subgroup in $G$. A {\it $p$-local subgroup} is the normalizer of a nontrivial $p$-subgroup in $G$.

\begin{dft}\label{defchar}
The group $G$ has {\it characteristic} $p$ if $C_G(O_p(G)) \leq O_p(G)$. If all $p$-local subgroups of $G$ have characteristic $p$ then $G$ has {\it local characteristic} $p$.
\end{dft}

We remark here that the notion of ``local characteristic $p$" is what group theorists usually call ``characteristic $p$-type".

\begin{prop}\label{chaprop}
Let $G$ be a finite group and let $p$ be a prime dividing its order.
\begin{list}{\upshape\bfseries}
{\setlength{\leftmargin}{.8cm}
\setlength{\labelwidth}{1cm}
\setlength{\labelsep}{0.2cm}
\setlength{\parsep}{0.5ex plus 0.2ex minus 0.1ex}
\setlength{\itemsep}{.6ex plus 0.2ex minus 0ex}}
\item[${\rm(i).}$] Assume $G$ has characteristic $p$. If $P$ is a $p$-subgroup of $G$ and $H$ a subgroup of $G$ with $PC_G(P) \le H \le N_G(P)$, then $H$ has characteristic $p$, see \cite[Lem. 1]{sol74}. In particular, $G$ has local characteristic $p$, see \cite[12.6]{gls2b}.
\item[${\rm(ii).}$] Let $P$ be a $p$-subgroup of $G$. The subgroup $C_G(P)$ has characteristic $p$ if and only if $N_G(P)$ has characteristic $p$, see \cite[5.12]{gls2b}.
\end{list}
\end{prop}

\begin{dft}\label{defpar}
A {\it parabolic subgroup} of $G$ is an overgroup of a Sylow $p$-subgroup of $G$. The group $G$ has {\it parabolic characteristic} $p$ if all $p$-local, parabolic subgroups of $G$ have characteristic $p$.
\end{dft}

We would like to mention that our usage of the term ``parabolic" differs from the classical notion of parabolics, in the special case of a Lie type group over a field of characteristic $p$, for in the latter, the parabolics are actually the overgroups not of just a Sylow $p$-subgroup $S$, but of a full Borel subgroup, i.e. of the normalizer $N_G(S)$. Our notion of ``parabolic characteristic $p$" appears elsewhere in group theory: notably for $p=2$ it appears as ``even characteristic" in Aschbacher and Smith \cite[p.3]{ascsm1}; this provides further evidence for the naturalness of this concept.

\subsection*{Collections of p-subgroups. Notation and known results}

A {\it collection} $\Cc$ of $p$-subgroups of $G$ is a set of $p$-subgroups which is closed under conjugation; a collection is a $G$-poset under the inclusion relation with $G$ acting by conjugation. The {\it order complex} or the {\it nerve} $\Delta(\Cc)$ is the simplicial complex which has as simplices proper inclusion chains from $\Cc$; for a more detailed discussion see \cite[2.7]{bs04}. If we let $|\Delta(\Cc)|$ to denote the geometric realization of the nerve $\Delta(\Cc)$, then the correspondence $\Cc \rightarrow |\Delta(\Cc)|$ allows assignment of topological concepts to posets \cite[Sect. 1]{qu78}. A collection $\mathcal{C}$ is contractible if $|\Delta(\Cc)|$ is contractible.

\begin{ntn}\label{fixpt}
Given a subgroup $P$ of $G$, let $\mathcal{C}^P = \lbrace Q \in \mathcal{C} \; | \; P \leq N_G(Q) \rbrace$ denote the subcollection of $\mathcal{C}$ fixed under the action of $P$. Next $\mathcal{C}_{>P} = \lbrace Q \in \mathcal{C} \; | \; P<Q \rbrace$. Similarly define $\mathcal{C}_{\geq P}$ and also $\mathcal{C}_{< P}$ and $\mathcal{C}_{\leq P}$. We will use the notation $\mathcal{C}_{>P}^{\leq H}$ for the set $\mathcal{C}_{>P}^{\leq H} = \lbrace Q \in \mathcal{C} \; | \;P<Q \leq H \rbrace$.
\end{ntn}

The following statement is a well known result, initially proved in a more general context, see for example \cite[8.2.4]{tdieck1} and \cite{jaseg}:

\begin{num}\label{equivfixed}
A poset map $f: \Cc \rightarrow \Cc'$ is a $G$-homotopy equivalence if and only if the induced map $f^H: \Cc^H \rightarrow \Cc'^H$ on $H$-fixed points is a homotopy equivalence for all subgroups $H$ of $G$.
\end{num}

\begin{thm}\label{homotopies}
Let $\Cc$ and $\Cc'$ be two collections of subgroups of $G$.
\begin{list}{\upshape\bfseries}
{\setlength{\leftmargin}{.8cm}
\setlength{\labelwidth}{1cm}
\setlength{\labelsep}{0.2cm}
\setlength{\parsep}{0.5ex plus 0.2ex minus 0.1ex}
\setlength{\itemsep}{.6ex plus 0.2ex minus 0ex}}
\item[${\rm(i).}$]\cite[Thm. 1]{tw91} Let $\Cc \subseteq \Cc'$. Assume either that $\Cc_{\geq P}$ is $N_G(P)$-contractible for all $P \in \Cc'$, or that $\Cc_{\leq P}$ is $N_G(P)$-contractible for all $P \in \Cc'$. Then the inclusion $\Cc \hookrightarrow \Cc'$ is a $G$-homotopy equivalence.
\item[${\rm(ii).}$]\cite[2.2(3)]{gs} Suppose that $F$ is a $G$-equivariant poset endomorphism of $\Cc$ satisfying either $F(P) \geq P$ or $F(P) \leq P$, for all $P \in \Cc$. Then, for any collection $\Cc'$ containing the image of $F$, the inclusions $F(\Cc) \subseteq \Cc' \subseteq \Cc$ are $G$-homotopy equivalences.
\item[${\rm(iii).}$]\cite[in proof Lem. 2.7(1)]{gs} Let $\Cc$ be a collection of $p$-subgroups that is closed under passage to $p$-overgroups. Let $Q$ be an arbitrary $p$-subgroup in $G$. Then the inclusion $\Cc_{\geq Q} \hookrightarrow \Cc^Q$ is a $N_G(Q)$-homotopy equivalence.
\end{list}
\end{thm}

\begin{thm}\cite[Thm. 1]{tw91}\label{hmap}
Let $G$ be a group and let $\varphi: \Cc \rightarrow \Cc'$ be a poset map between two collections of subgroups of $G$. Suppose that for all $P \in \Cc'$, the subcollection $\varphi^{-1}(\Cc'_{\ge P})$ is $N_G(P)$-contractible. Then $\varphi$ is a $G$-homotopy equivalence.
\end{thm}

We end this succinct review with a result due to Grodal \cite[pp. 420-421]{gr02}, see also \cite[Thm. 1]{sa06} or for an easy proof \cite[Prop. 4.3]{mgo4}.

\begin{prop}\label{grosubcol}
Let $\Cc$ be a collection of nontrivial $p$-subgroups of $G$, which is closed under passage to $p$-overgroups. Let $Q$ be a $p$-subgroup of $G$. If $\Cc '  \subseteq \Cc$ is a subcollection
which contains every $p$-radical subgroup in $\Cc$  then $\Cc_{> Q}$ is $N_G(Q)$-homotopy equivalent to $\Cc'_{>Q}$.
\end{prop}

\subsection*{Standard collections of p-subgroups}

A $p$-subgroup $R$ of $G$ is called $p$-{\it radical} if $R = O_p(N_G(R))$. Every $p$-subgroup $Q$ of $G$ is contained in a $p$-radical subgroup of $G$ uniquely determined by $Q$ and $G$. This is called the {\it radical closure} of $Q$ in $G$ and is the last term $R_Q$ of the chain $P_{i+1} = O_p(N_G(P_i))$ starting with $Q = P_0$. It is easy to see that $N_G(Q) \leq N_G(R_Q)$. A $p$-subgroup $R$ is called $p$-{\it centric} if $Z(R)$ is a Sylow $p$-subgroup of $C_G(R)$, in which case $C_G(R) = Z(R) \times H$, with $H$ a subgroup of order relatively prime to $p$. If $R$ is $p$-centric and $Q$ is a $p$-subgroup of $G$ which contains $R$ then $Q$ is also $p$-centric and $Z(Q) \le Z(R)$.

In what follows $\Ss (G)$ will denote the Brown collection of nontrivial $p$-subgroups and $\Bb(G)$ will denote the Bouc collection of nontrivial $p$-radical subgroups\footnote{It is customary to denote these collections by $\Ss_p(G)$ and $\Bb_p(G)$, but since the prime $p$ is fixed throughout the first sections of the paper, in order to simplify the notation we will drop the subscript $p$.}. The inclusion $\Bb(G) \subseteq \Ss(G)$ is a $G$-homotopy equivalence \cite[Thm. 2]{tw91}. Let $\Dd(G)$ denote the subcollection of $\Ss(G)$ consisting of the nontrivial $p$-centric and $p$-radical subgroups of $G$. This collection is not always homotopy equivalent with $\Ss(G)$.

\subsection*{Distinguished collections of p-subgroups}

We recollect in this section the definitions and some facts on certain subfamilies of $\Ss(G)$, called distinguished collections of $p$-subgroups. These collections were introduced in \cite{ mgo1} and further analyzed in \cite{mgo3} and in \cite{mgo4}.

\begin{num}\label{pcen}
An element $x$ of order $p$ in $G$ is {\it $p$-central} if $x$ is in the center of a Sylow $p$-subgroup of $G$. Let $\widehat{\Gamma}_p(G)$ denote the family of $p$-central elements of $G$.
\end{num}

\begin{dft}\label{hatdef}
For $\Cc(G)$ a collection of $p$-subgroups of $G$ denote:
$$\wCc(G) = \lbrace P \in \Cc(G) \; | \; Z(P) \cap \widehat{\Gamma}_p(G) \not= \emptyset \; \rbrace$$
the collection of subgroups in $\Cc(G)$ which contain $p$-central elements in their centers. We call $\wCc(G)$ the {\it distinguished} $\Cc(G)$ {\it collection}. We shall refer to the subgroups in $\wCc(G)$ as {\it distinguished subgroups}. Also, denote
$$\tCc(G) = \lbrace P \in \Cc(G) \; | \; P \cap \widehat{\Gamma}_p(G) \not= \emptyset \; \rbrace$$
the collection of subgroups in $\Cc(G)$ which contain $p$-central elements. Obviously, the following inclusions hold $\wCc(G) \subseteq \tCc(G) \subseteq \Cc(G)$.
\end{dft}

\begin{num}\label{tildeclosed}
For later reference, we collect a few useful facts. The proofs are short and elementary, details can be found in \cite[Lem. 3.3, Lem. 3.5 and Prop. 3.7]{mgo4}.
\begin{list}{\upshape\bfseries}
{\setlength{\leftmargin}{1.4cm}
\setlength{\labelwidth}{1cm}
\setlength{\labelsep}{0.2cm}
\setlength{\parsep}{0cm}
\setlength{\itemsep}{0cm}}
\item[${\rm(i).}$] The collection $\tSs(G)$ is closed under passage to $p$-overgroups.
\item[${\rm(ii).}$] Let $P \in \Ss(G)$. If $Q \in \wSs(G)_{>P}$ then $N_Q(P) \in \wSs(G)$.
\item[${\rm(iii).}$] If $N_G(P)$ has characteristic $p$, then $O_p(N_G(P))$ is $p$-centric and distinguished.
\item[${\rm(iv).}$] If $G$ has parabolic characteristic $p$ and if $P \in \tSs(G)$ then $N_G(P)$ has characteristic $p$.
\end{list}
\end{num}

The distinguished collections of $p$-subgroups exhibit similar behavior as their standard counterparts. Some of these properties are collected in the following:

\begin{thm}\label{propdist}
Let $G$ be a finite group.
\begin{list}{\upshape\bfseries}
{\setlength{\leftmargin}{.8cm}
\setlength{\labelwidth}{1cm}
\setlength{\labelsep}{0.2cm}
\setlength{\parsep}{0.1cm}
\setlength{\itemsep}{0cm}}
\item[${\rm(i).}$]\cite[Prop. 3.4]{mgo4} All $p$-centric subgroups of $G$ are distinguished, hence $\Dd(G) \subseteq \wBb(G)$.
\item[${\rm(ii).}$]\cite[Prop. 3.6]{mgo4} If $G$ has local characteristic $p$, then $\Dd(G)=\wBb(G)=\Bb(G)$.
\item[${\rm(iii).}$]\cite[Prop. 3.7]{mgo4} If $G$ has parabolic characteristic $p$, then $\Dd(G) = \wBb(G) = \widetilde{\Bb}(G)$.
\item[${\rm(iv).}$]\cite[Prop. 3.9, 3.10]{mgo4} If $G$ has parabolic characteristic $p$, then $\wBb(G)$, $\wSs(G)$ and $\tSs(G)$ are $G$-homotopy equivalent.
\end{list}
\end{thm}

\begin{num}
In the case that $G$ has parabolic characteristic $p$, the better known collection of $p$-centric $p$-radical subgroups is equal to the collection of distinguished $p$-radical subgroups. The reason for using the ``distinguished" point of view for this collection resides in the fact that our approach is far more useful and easier to work with in proofs.
\end{num}

\section{The structure of the fixed point sets}

We investigate the fixed point sets of $p$-elements of $G$ acting on the order complex of the collection $\Dd(G)$ of $p$-centric and $p$-radical subgroups. We are interested in groups of parabolic characteristic $p$, in which case, as seen in Theorem \ref{propdist}, the $p$-centric and $p$-radical subgroups are the distinguished $p$-radical subgroups of $G$ and the collections $\wBb(G)$, $\wSs(G)$ and $\tSs(G)$ are $G$-homotopy equivalent.

\begin{ntn}\label{cnotat}
Throughout this section, $G$ is a finite group of parabolic characteristic $p$, where $p$ is a prime dividing its order. Let $T$ be a $p$-subgroup of $G$; we use the shorthand notation $C = T \cdot C_G(T)$ and $O_C = O_p(T \cdot C_G(T))$.
\end{ntn}

\begin{lem}\label{lem>M}
Let $G$ be a finite group of parabolic characteristic $p$. The inclusion $\mathcal{X} := \wSs (G)^{\leq C}_{>M} \hookrightarrow \mathcal{Y} := \tSs(G)^{\leq C}_{>M}$ is a $N_G(T)$-homotopy equivalence, for every $p$-subgroup $M$ that satisfies $T \leq M \leq C$ and $N_G(T) \leq N_G(M)$.
\end{lem}

\begin{proof}
Let $P \in \mathcal{Y}$; we will show that $\mathcal{X}_{\geq P}$ is equivariantly contractible and apply Theorem \ref{homotopies}(i). Let $R_P$ be the radical closure of $P$ and let $Q \in \mathcal{X}_{\geq P}$. Consider the string of poset maps given by:
$$Q \geq N_Q(P) \leq N_Q(P) Z(R_P) \geq P Z(R_P)$$
Note that $T \leq M < P \leq R_P$, so that $Z(R_P) \leq C$. By \ref{tildeclosed}(iii) and (iv), $R_P$ is $p$-centric and distinguished. This implies that $Z(R_P)$ is distinguished, and also $P Z(R_P)$ is distinguished, since $Z(R_P) \leq Z(P Z(R_P))$. Next \ref{tildeclosed}(ii) implies that $N_Q(P)$ is distinguished. Observe that $N_Q(P) \leq N_G(P) \leq N_G(R_P)$, implying $N_Q(P) Z(R_P)$ is a group. Choose $S_R \in \syl_p(N_G(R_P))$ satisfying $N_Q(P) \leq S_R$, and extend it to $S \in \syl_p(G)$. Note $R_P \leq S_R \leq S$. Then $Z(S) \leq C_G(N_Q(P))$ and $Z(S) \leq C_G(R_P)=Z(R_P)$, since $R_P$ is $p$-centric and $p$-radical. This implies that $Z(S) \leq Z(N_Q(P)Z(R_P))$, so that $N_Q(P)Z(R_P)$ is distinguished. Thus the above is a string of equivariant poset maps $\mathcal{X}_{\geq P} \rightarrow \mathcal{X}_{\geq P}$, which proves the equivariant contractibility of the subcollection $\mathcal{X}_{\geq P}$. Note the second condition on $M$ is needed so that $N_G(T)$ acts on $\mathcal{X}$ and $\mathcal{Y}$.
\end{proof}

The next result was first obtained for $T$ a cyclic group of order $p$ in \cite{mgo4}, and the arguments from Propositions $4.4$, $4.5$ and $4.7$ of that paper also apply to any $p$-group $T$. For completeness we include these proofs below. The notations from \ref{cnotat} are maintained.

\begin{thm}\label{ocintilde}
Let $G$ be a group of parabolic characteristic $p$ and let $T$ be a $p$-subgroup of $G$. If $O_C \in \tSs(G)$ then the fixed point set $\Dd(G)^T$ is $N_G(T)$-contractible.
\end{thm}

\begin{proof}
The result is obtained via a string of homotopy equivalences.

{\it Step 1}: $\Dd(G)^T = \wBb(G)^T \hookrightarrow \wSs(G)^T \hookrightarrow \tSs(G)^T$ are $N_G(T)$-homotopy equivalences, since $\Dd(G) = \wBb(G) \hookrightarrow \wSs(G) \hookrightarrow \tSs(G)$ are $G$-homotopy equivalences by Theorem \ref{propdist}(iii) and (iv).

{\it Step 2}: $\tSs(G)_{\geq T} \hookrightarrow \tSs(G)^T$ is an $N_G(T)$-homotopy equivalence by Theorem \ref{homotopies}(iii) and \ref{tildeclosed}(i).

{\it Step 3}: If $T \in \tSs(G)$ then $\tSs(G)_{\geq T}$ is contractible, a cone on $T$, and we are done. Thus we assume $T \not\in \tSs(G)$, so that $\tSs(G)_{\geq T} = \tSs(G)_{>T}$.

{\it Step 4}: $\wSs(G)_{>T} \hookrightarrow \tSs(G)_{>T}$ is an $N_G(T)$-homotopy equivalence by Proposition \ref{grosubcol}. We are also using \ref{tildeclosed}(i) and Theorem \ref{propdist}(iii). The latter implies $\tSs(G) \cap \Bb(G) \subseteq \wSs(G)$.

{\it Step 5}: $\wSs(G)^{\leq C}_{>T} \hookrightarrow \wSs(G)_{>T}$ is an $N_G(T)$-homotopy equivalence by Theorem \ref{homotopies}(ii), where the poset map used is $F(P) = P \cap C$, for $P \in \wSs(G)_{>T}$. Since $T < P$, we have $Z(P) \leq C_G(T) \leq C$, so that $Z(P) \leq P \cap C$.

{\it Step 6}: The inclusion $\wSs(G)^{\leq C}_{>T} \hookrightarrow \tSs(G)^{\leq C}_{>T}$ is an $N_G(T)$-homotopy equivalence by Lemma \ref{lem>M}, applied with $M = T$.

{\it Step 7}: $\tSs(G)^{\leq C}_{\ge O_C} \hookrightarrow \tSs(G)^{\leq C}_{>T}$ is an $N_G(T)$-homotopy equivalence by Theorem \ref{homotopies}(ii), with poset map $F(P)=P \cdot O_C$ for $P \in \tSs(G)^{\le C}_{>T}$. Since we have assumed that $O_C \in \tSs(G)$, we have $\tSs(G)^{\le C}_{\ge O_C}$ is contractible, a cone on $O_C$. Thus we have shown that $\Dd^T(G)$ is $N_G(T)$-contractible.
\end{proof}

We arrive at the main result of this section. It asserts, that under certain assumptions on the shape of the centralizer of a purely noncentral $p$-subgroup $T$ of $G$, the structure of the fixed point set $\Dd^T(G)$ is determined by the collection of $p$-centric and $p$-radical subgroups in a subquotient (or section) of this centralizer. Our Theorem \ref{noncenthm} generalizes the result of \cite[Thm. 4.14]{mgo4}. Before stating and proving the theorem, we need to introduce some more notation.

\begin{ntn}\label{triplentn}
Let $T$ be any $p$-subgroup of $G$, and denote $C = T \cdot C_G(T)$ and $O_C = O_p(T \cdot C_G(T))$, maintaining the notation introduced in \ref{cnotat}. Let $L \trianglelefteq C$ be a normal subgroup of $C$ with $O_C \leq L$. Denote the quotient groups by $K = C/L$ and $H = L/O_C$, so we can write $C=O_C.H.K = L.K$.\\
Let $S_C \in \syl_p(C)$ and extend it to $S \in \syl_p(G)$. Since $T \le S_C \le S$ it follows that $Z(S) \le Z(S_C)$. Next, define $S_L = S_C \cap L$, a Sylow $p$-subgroup of $L$. As $S_L \trianglelefteq S_C$, $Z(S_C) \cap S_L \ne 1$. However, this does not guarantee that $Z(S) \cap S_L \ne 1$. The sequence of subgroups $S_L \trianglelefteq S_C \le S$ chosen as above will be denoted $(S_L, S_C, S)$ and will be referred to as a {\it triple of Sylow $p$-subgroups}.
\end{ntn}

\begin{num}
We warn the reader that the ``hat" symbol in $\wSs(G)$ and $\wSs (H)$ has different meanings; in the former it refers to those groups that are distinguished in $G$, while in the latter it refers to subgroups that are distinguished in $H$.
\end{num}

\begin{thm}\label{noncenthm}
Let $G$ be a finite group of parabolic characteristic $p$, and let $T$ be a $p$-subgroup of $G$. Assume the following hold:
\begin{list}{\upshape\bfseries}
{\setlength{\leftmargin}{1.2cm}
\setlength{\labelwidth}{1cm}
\setlength{\labelsep}{0.2cm}
\setlength{\parsep}{0.1cm}
\setlength{\itemsep}{0cm}}
\item[${\rm(N1).}$] The group $O_C$ is purely noncentral in $G$.
\item[${\rm(N2).}$] The group $T\cdot C_G(T)$ has the form $C=O_C.H.K$, the group $H$ has parabolic characteristic $p$, and $L = O_C.H$ is a normal subgroup of $N_G(T)$.
\item[${\rm(N3).}$] There exists a triple of Sylow $p$-subgroups $(S_L, S_C, S)$ with the property that $Z(S)\cap S_L \ne 1$.
\end{list}
Then there is a $N_G(T)$-equivariant homotopy equivalence between $\Dd(G)^T$ and $\Dd(H)$.
\end{thm}

\begin{num}\label{n3all}
The condition (N3) implies that if $S'_C \in \syl_p(C)$, then there exists $S' \in \syl_p(G)$ such that $Z(S') \cap (S'_C \cap L) \ne 1$. There exists an element $g \in C$ with $S'_C = g S_C g^{-1}$. Define $S' = g S g^{-1}$. Since $L \trianglelefteq C$, $S'_L = S'_C \cap L = g(S_C \cap L)g^{-1}$. Thus $Z(S') \cap S'_L = g( Z(S) \cap S_L )g^{-1}$.
\end{num}

There are several examples involving sporadic finite simple groups in which a condition stronger than (N3) is satisfied: all $p$-central involutions of $G$ which lie in $C$ actually lie in $L$. This condition was used to prove a similar theorem; see \cite[Thm. 2.1 and Lem. 2.3]{mgo5}.

\begin{num}\label{landoc}
In some examples we will consider in Section 6, we have $L = C_G(O_C)$. Then whenever we have a triple $(S_L,S_C,S)$ of Sylow $p$-subgroups, we will have $Z(S) \leq L$ since $O_C \leq S$. Thus (N3) is satisfied. Also, (N3) will be satisfied if we know that $Z(S_C) \leq L$, which will be true in most cases.
\end{num}

There is no direct relationship between the elements which are $p$-central in $G$ and those which are $p$-central in $C$, or in $L$. In order to deal with this difficulty, we introduce the following subcollection of $\widehat{\mathcal{S}}(G)^{\leq C}_{>O_C}$.

\begin{num}\label{tcol}
Use the notations from \ref{fixpt}, \ref{cnotat} and \ref{triplentn} and set:
$$\mathcal{T} = \lbrace P \in \wSs(G)^{\leq C} _{>O_C} \; \Big| \; Z(P) \cap Z(S) \cap Z(S_L) \not = 1, \; \text{for some triple} \; (S_L,S_C,S) \; {\rm with} \; P \leq S_C \leq S \rbrace.$$
\end{num}

This collection is contained in $\wSs(C)$, and it contains all the $p$-centric subgroups $P$ in $C$ which properly contain $O_C$, since if $P$ is $p$-centric and $P \leq S_C \leq S$, then $Z(S) \leq C_G(P)$, implying $Z(S) \leq Z(P)$. Another collection of interest will be $\mathcal{T}^{\leq L}$; the $p$-subgroups it contains are distinguished as subgroups of $G$ and also are distinguished as subgroups of $L$. The condition $N_G(T) \leq N_G(L)$ in (N2) is needed so that $N_G(T)$ acts on $\mathcal{T}^{\le L}$. This condition will be satisfied if $L$ is the generalized Fitting subgroup of $C$ (there exists one component $E$ of $C$ with $E / Z(E) = H$).

\begin{num}\label{normt}
Let $P \in \mathcal{T}$ and assume that $O_C < Q \le P$. Then $N_P(Q) \in \mathcal{T}$.\\
It is easy to see that if $Q \le N_P(Q) \le P$ then $Z(P) \le Z(N_P(Q))$ and this suffices to show $N_P(Q) \in \mathcal{T}$.
\end{num}

\begin{proof}[Proof of Theorem \ref{noncenthm}]
The proof will consist of five steps and a number of homotopy equivalences between various intermediary collections. For the rest of the proof we assume that $G$ and $T$ satisfy all the conditions from the hypotheses of Theorem \ref{noncenthm}.

\hspace*{1cm}{\it Step 1: $\Dd(G)^T$ and $\wSs(G)^{\le C}_{>O_C}$ are $N_G(T)$-homotopy equivalent.}

The proof of Theorem \ref{ocintilde} shows that there is an $N_G(T)$-homotopy equivalence between $\Dd(G)^T$ and $\tSs(G)^{\le C}_{\ge O_C} = \tSs(G)^{\le C}_{>O_C}$, using that $O_C \not\in \tSs(G)$. Then an application of Lemma \ref{lem>M} with $M = O_C$ gives an equivalence with $\wSs(G)^{\le C}_{>O_C}$.

\hspace*{1cm}{\it Step 2: The inclusion $\mathcal{T} \hookrightarrow \wSs(G)^{\leq C} _{>O_C}$ is a $N_G(T)$-homotopy equivalence.}

We will apply Theorem \ref{homotopies}(i) once again. We need to show that $\mathcal{T}_{\geq Q}$ is equivariantly contractible for every $Q \in \widehat{\mathcal{S}}(G)^{\leq C} _{>O_C}$. The equivariance here is with respect to the action of $N_G(T) \cap N_G(Q)$. Set $O_{CQ}:=O_p(N_C(Q))$. Assume now that $P \in \mathcal{T}_{\geq Q}$ and consider the contracting homotopy given by the following string of equivariant poset maps:
$$P \geq N_P(Q) \leq N_P(Q) O_{CQ} \geq O_{CQ}.$$
In order to complete the proof we need to prove that each of the subgroups involved in the above string lies in $\mathcal{T}_{\geq Q}$. Since $Q$ is a distinguished $p$-subgroup of $G$, $N_G(Q)$ has characteristic $p$, by \ref{tildeclosed}(iv). Next, because $T < Q \leq C$ it follows that $C_G(Q) \leq C_G(T) \le C$ and thus $QC_G(Q) \leq N_{C}(Q) \leq N_G(Q)$, and by Proposition \ref{chaprop}(i), $N_{C}(Q)$ has characteristic $p$. Hence $C_{N_C(Q)}(O_{CQ})\le O_{CQ}$. But $Q \le O_{CQ}$ so $C_G(O_{CQ}) \le C_G(Q) \le N_C(Q)$ and therefore $C_G(O_{CQ}) = Z(O_{CQ})$ which implies $Z(S_C) \le Z(O_{CQ})$ for any $S_C \in \syl_p(C)$ which contains $O_{CQ}$. Extend $S_C$ to a Sylow $p$-subgroup $S$ of $G$ satisfying $Z(S) \cap S_L \ne 1$, using (N3) and \ref{n3all}. Recall that $Z(S) \leq Z(S_C)$. Since $Z(S) \cap Z(S_L) \leq Z(S_C) \leq O_{CQ}$, it follows that $O_{CQ} \in \mathcal{T}$. The fact that $N_P(Q) \in \mathcal{T}_{\ge Q}$ follows from \ref{normt}. Next consider $N_P(Q)O_{CQ}$ which is a $p$-subgroup of $N_C(Q)$ and which contains $Z(S_C)$ for any $S_C \in \syl_p(C)$ which contains it. Thus all the groups in the string are in $\mathcal{T}_{\ge Q}$.

\hspace*{1cm}{\it Step 3: The inclusion $\mathcal{T}^{\le L} \hookrightarrow \mathcal{T}$ is a $N_G(T)$-homotopy equivalence.}

Apply \ref{homotopies}(ii) with poset map $F(P)=P \cap L$ for $P \in \mathcal{T}$. Note that $Z(P) \cap Z(S) \cap Z(S_L) \le Z(P \cap L)$.

{\it Step 4:} Recall that $L \simeq O_C.H$ and the quotient group $H \simeq L/O_C$ has parabolic characteristic $p$. The quotient map will be denoted $q : L \rightarrow H$. For $M \leq L$, let $\overline{M} = q(M)$. For $Q \leq L$, denote $O_{LQ} = O_p(N_L(Q))$; for $\overline{Q} \leq H$, denote $O_{HQ} = O_p(N_H(\overline{Q}))$. If $(S_L,S_C,S)$ is a triple of Sylow $p$-subgroups defined as in \ref{triplentn}, we let $\overline{S}_L=q(S_L) \simeq S_H \in \syl_p(H)$. We prove the following:
\par
\rightskip=1cm
\leftskip=.5cm
{\it The poset map $q_*: \mathcal{T}^{\le L} \rightarrow \wSs(H)$ induced by the quotient map $q: L \rightarrow H$ is an equivariant homotopy equivalence.}
\par
\leftskip=0cm
\rightskip=0cm

We first show that $q_*(\mathcal{T}^{\le L}) \subseteq \wSs(H)$. Since $O_C$ is purely noncentral in $G$, the map $q : L \rightarrow H$ is injective on the elements of $Z(S)$ as they are $p$-central in $G$. Therefore, if $P \in \mathcal{T}^{\le L}$ then $Z(P) \cap Z(S) \cap Z(S_L) \not = 1$ for some triple of Sylow $p$-subgroups $(S_L,S_C,S)$ with $P \le S_L$, and this implies $Z(\overline{P}) \cap Z(\overline{S}_L) \not = 1$, and we have $q_*(P)= \overline{P} \in \wSs(H)$.

According to Theorem \ref{hmap}, the poset map $q_*: \mathcal{T}^{\le L} \rightarrow \widehat{\mathcal{S}}(H)$ is an equivariant homotopy equivalence if $q_*^{-1}(\widehat{\mathcal{S}}(H)_{\geq {\overline{Q}}})$ is equivariantly contractible for any $\overline{Q} \in \widehat{\mathcal{S}}(H)$.

Define $Q = q^{-1}(\overline{Q})$. Recall that $O_C \le P$, thus we have:
$$q_*^{-1}(\widehat{\mathcal{S}}(H)_{\geq \overline{Q}}) = \lbrace P \in \mathcal{T}^{\le L} \;| \;\overline{Q} \leq \overline{P} \rbrace = \lbrace P \in \mathcal{T}^{\le L} \;|\; Q \leq P \rbrace = \mathcal{T}^{\le L}_{\ge Q}\;.$$

For $P \in \mathcal{T}^{\le L}_{\geq Q}$, consider the string of equivariant poset maps given by:
$$P \geq N_P(Q) \leq N_P(Q) O_{LQ} \geq O_{LQ}\;.$$

We need to show that all of these terms lie in $\mathcal{T}^{\le L}_{\geq Q}$. The fact that $N_P(Q) \in \mathcal{T}^{\le L}_{\ge Q}$ follows from \ref{normt}. Next, we have $N_L(Q) = q^{-1}(N_H(\overline{Q}))$, using $O_C \leq Q$. Thus $O_{LQ}$ is equal to $q^{-1}(O_{HQ})$ by the correspondence theorem for normal subgroups applied to $N_L(Q) \rightarrow N_H (\overline{Q})$. Since $\overline{Q} \in \widehat{\mathcal{S}}(H)$, $N_H(\overline{Q})$ has characteristic $p$, by \ref{tildeclosed}(iv) and using our assumption that $H$ has parabolic characteristic $p$. Thus $C_H(O_{HQ}) \leq O_{HQ}$ and $N_H(O_{HQ})$ also has characteristic $p$, which follows by an application of Proposition \ref{chaprop}(ii). Next, we have $C_L(O_{LQ}) \le q^{-1}(C_H(O_{HQ})) \le q^{-1}(O_{HQ}) = O_{LQ}$, which shows that the group $O_{LQ}$ is $p$-centric in $L$. It follows that $Z(S_L) \le Z(O_{LQ})$ for every Sylow $p$-subgroup of $L$ which contains $O_{LQ}$. But by our assumption ${\rm(N3)}$, and using \ref{n3all}, the subgroup $Z(S_L)$ contains $p$-central elements of $G$; therefore it follows that $O_{LQ}$ is distinguished in $G$. Consequently $O_{LQ}$ lies in $\mathcal{T}^{\le L}_{\ge Q}$. Now consider $N_P(Q)O_{LQ}$ which is a subgroup of $N_L(Q)$. Since $O_{LQ}$ lies in every Sylow $p$-subgroup of $N_L(Q)$ and $P \in \mathcal{T}^{\le L}_{\ge Q}$ it follows that the nontrivial elements in $Z(P) \cap Z(S) \cap Z(S_L)$ also lie in $Z(N_P(Q)O_{LQ})$ and the subgroup $N_P(Q)O_{LQ}$ is indeed in $\mathcal{T}^{\le L}_{\ge Q}$ (that is $O_{LQ} \le S_L$, with $S_L$ the Sylow $p$-subgroup which contains $P$ and satisfies the required condition from \ref{tcol}).

\hspace*{1cm}{\it Step 5: There is an equivariant homotopy equivalence between $\Dd (G)^T$ and $\Dd (H)$}.

Combine the first four steps of the proof to obtain the chain of $N_G(T)$-homotopy equivalences:
$$\Dd(G)^T \simeq \wSs(G)^{\leq C}_{>O_C} \simeq \mathcal{T} \simeq \mathcal{T}^{\leq L} \simeq \wSs (H).$$
Recall that $H$ has parabolic characteristic $p$ and therefore $\wSs (H)$ and $\Dd (H)$ are $H$-homotopy equivalent, by Proposition \ref{propdist}(iii)-(iv). The proof in \cite{mgo4} of this homotopy equivalence can actually be seen to yield equivariance under any automorphism of $H$, so that here we have equivariance under $N_G(T) / O_C$, or under $N_G(T)$ with $O_C$ acting trivially.
\end{proof}

\section{Some modular representation theory}

\subsection*{On the group algebra of a group of parabolic characteristic $\mathbf{p}$}

Let $G$ be a finite group and let $\fk$ be a field of characteristic $p$, which is a splitting field for all the subgroups of $G$. Denote by $\fk G$ the group algebra of $G$ over the field of coefficients $\fk$. We are interested in the block structure of $\fk G$ and in particular its defect groups. For details on the background material used in this section the interested reader is referred to \cite{alperin}, or for a more succinct exposition to \cite{ben91a}.

Although not all the results from this Section are needed in the later proofs, the information given here provides valuable insight into the properties of the group algebras of those classes of groups considered in the next two Sections.

\begin{num}\label{charkH}
If a finite group $G$ has characteristic $p$, then $C_G(O_p(G)) \le O_p(G)$ and $\fk G$ has only one block \cite[Prop. 6.2.2]{ben91a}.
\end{num}

It is now easy to prove that:

\begin{lem}\label{localcharkH}
Let $G$ be a finite group of local characteristic $p$. Then every nonprincipal block of $\fk G$ has defect zero.
\end{lem}

\begin{proof}
If $B$ is a block of $\fk G$ with a nontrivial defect group $D$, then by Brauer's First Main Theorem there is a corresponding block $b$ of $\fk N_G(D)$. Since $G$ has local characteristic $p$, the subgroup $N_G(D)$ has characteristic $p$ and thus the only block is the principal block, by \ref{charkH}. Then Brauer's Third Main Theorem implies that $B$ is the principal block of $\fk G$.
\end{proof}

\begin{lem}\label{blpar}
Assume that a finite group $G$ has parabolic characteristic $p$. If $D$ is a defect group of a nonprincipal block of $\fk G$, then $D$ is purely noncentral.
\end{lem}

\begin{proof}
Let $B$ be a block of $\fk G$ with a nontrivial defect group $D$, and let $b$ denote its Brauer correspondent in $\fk N_G(D)$. If $D$ contains a $p$-central element, then according to \ref{tildeclosed}(iv), $N_G(D)$ has characteristic $p$. Thus $b$ is the principal block of $\fk N_G(D)$, and $B$ is the principal block of $\fk G$.
\end{proof}

\begin{num}\label{defectpar}
If $G$ has parabolic characteristic $p$, the only distinguished $p$-radical subgroup which is a defect group of $\fk G$ is the Sylow $p$-subgroup of $G$, and the only block with this defect group is the principal block of $\fk G$.
\end{num}

Next, we relate the defect groups of $\fk G$ to the defect groups of the group algebra $\fk C_G(t)$ of the centralizer of an element $t$ of order $p$.

\begin{lem}\label{defectcen}
Let $G$ be a finite group and let $D$ be a defect group of a block of $\fk G$. If $T \le Z(D)$ then $D$ is a defect group of $\fk C_G(T)$.
\end{lem}

\begin{proof}
Let $C$ denote $TC_G(T)=C_G(T)$. Since $T \le Z(D)$ it follows that $D C_G(D) \leq N_G(D) \cap C \leq N_G(D)$. By Brauer's Extended First Main Theorem, $\fk N_G(D)$ has a block $B$ with defect group $D$, and $\fk DC_G(D)$ has a block $b$ with defect group $D$, with $b^{N_G(D)} = B$.

The Brauer correspondents $b^{N_C(D)}$ and $b^{N_G(D)}$ are defined, and $b^{N_G(D)} = (b^{N_C(D)})^{N_G(D)}$. There is a defect group $D'$ of $b^{N_C(D)}$ containing the defect group $D$ of $b$, and up to conjugacy $D'$ is contained in the defect group $D$ of $b^{N_G(D)}$; see \cite[Lem. 14.1(1)]{alperin}. Thus $D' = D$ and so $\fk N_C(D)$ has a block $b^{N_C(D)}$ with defect group $D$. Another application of Brauer's First Main Theorem to $N_C(D) \le C$ gives that $\fk C$ has a block $b^C$ with defect group $D$.
\end{proof}

\begin{num}
If $t$ is an element of order $p$ and $T := \langle t \rangle$, the index of $C_G(t)$ in $N_G(T)$ is relatively prime to $p$. Since $DC_G(D) \leq C_G(t) \trianglelefteq N_G(T)$, an application of \cite[Thm. 15.1(2)(4)(5)]{alperin} tells us that the defect groups of $\fk N_G(T)$ are the same as the defect groups of $\fk C_G(t)$.
\end{num}

Lemma \ref{defectcen} does not give any result regarding a defect group for $\fk G$, when a defect group for $\fk C_G(t)$ is given. However, the following consequence of \cite[Thm. 5.5.21]{nt} is useful.

\begin{num}
Let $t$ be an element of order $p$ in $G$ and set $C = C_G(t)$. Let $b$ be a block of $\fk C$ with defect group $Q$. The group $T = \langle t \rangle$ is normal in $C$, and since $Q$ is a $p$-radical subgroup of $C$, it follows that $T \trianglelefteq Q$ and $C_G(Q) \le C$. Thus $b^G$ is defined. Then for a suitable defect group $D$ of $b^G$, $Z(D) \le Z(Q) \le Q \le D$ and $Q = D \cap C$.
\end{num}

We record the following direct consequence of the above discussion:

\begin{prop}\label{noncen}
Let $G$ be a finite group of parabolic characteristic $p$, and let $D$ be a defect group of a nonprincipal block of $\fk G$. For any $t \in Z(D)$, the subgroup $O_p(C_G(t))$ is purely noncentral.
\end{prop}

\begin{proof}
The subgroup $D$ is purely noncentral by Lemma \ref{blpar}, and according to Lemma \ref{defectcen} the subgroup $D$ is also a defect group of a block of $\fk C_G(t)$. Thus $O_p(C_G(t)) \leq D$.
\end{proof}

\subsection*{The group ring of $\mathbf{O_C.H.K}$}

Assume that $C$ is a finite group of the form $C = O_C.H.K$, with $O_C = O_p(C)$. Set $\tH = C/O_C = H.K$. We also assume that $K = \tH / H$ is either a $p$-group or a $p'$-group, and that $H$ is either a group of Lie type in characteristic $p$ or a group of local characteristic $p$. These cases will be significant to the study of those groups we are interested in.

\begin{lem}\label{defcprime}
Let $C = O_C.H.K$ with $O_C = O_p(C)$, and assume that $K$ is a $p'$-group and that $H$ has local characteristic $p$. The defect groups of blocks of $\fk C$ are either equal to $O_C$ or are Sylow $p$-subgroups of $C$.
\end{lem}

\begin{proof}
Let $B$ be a block of $\fk C$ with defect group $D$. Since $O_C$ is a normal $p$-subgroup, $O_C \leq D$, and there exists a block $\widetilde{b}$ of $\tH$ dominated by $B$ (or contained in $B$) with a defect group equal to $D / O_C$ \cite[Thm 9.9(b)]{n98}. Next, there exists a block $b$ of $\fk H$ which is covered by the block $\widetilde{b}$. If $D'$ is a defect group of $b$, then there exists a defect group $\widetilde{D}$ of $\widetilde{b}$ such that $D' = \widetilde{D} \cap H$ \cite[Thm. 15.1(2)]{alperin}. Since $K$ is a $p'$-group, we must have $D' = \widetilde{D}$, which up to conjugacy equals $D / O_C$. We have assumed that $H$ has local characteristic $p$, and all of the nonprincipal blocks of $\fk H$ have defect zero. So either $D = O_C$ or $D / O_C \in \syl_p(H)$ and $D \in \syl_p(C)$.
\end{proof}

\begin{lem}\label{defcp}
Let $C = O_C.H.K$ with $O_C = O_p(C)$. Assume $K$ is a $p$-group and that $H$ is a finite simple group of Lie type in characteristic $p$. Then $\fk C$ has at most two blocks, and if a nonprincipal block exists, it has defect group of the form $D = O_C.K$.
\end{lem}

\begin{proof}
Note that $\fk H$ has exactly two blocks, the principal block $b_0$ and a block $b_1$ of defect zero (containing the Steinberg module). These are both invariant under the action of $\tH$. Since $K$ is a $p$-group, each block of $\fk H$ is covered by a unique block of $\tH$ \cite[Cor. 9.6]{n98}. So $\fk \tH$ has exactly two blocks, the principal block $\widetilde{b}_0$ (covering $b_0$) and one nonprincipal block $\widetilde{b}_1$ (covering $b_1$). The defect group $\widetilde{D}$ of the nonprincipal block $\widetilde{b}_1$ satisfies $\widetilde{D} \cap H = 1$ and has order equal to the order of $K$ \cite[Thm. 15.1(2)(4)]{alperin}. Thus $\widetilde{D}$ is isomorphic to $K$. (Note this shows $\tH = H:K$ is a split extension.) Finally, since $O_C$ is a normal $p$-subgroup of $C$, each block $B$ of $\fk C$ dominates some block of $\fk \tH$ and each block of $\fk \tH$ is dominated by exactly one block of $\fk C$. Thus either $\fk C$ has one block only (dominating both blocks of $\fk \tH$), or $\fk C$ has two blocks, the principal block $B_0$ and one nonprincipal block $B_1$ with defect group $D$. We have $O_C \leq D$ and $D / O_C = \widetilde{D} \simeq K$.
\end{proof}

\begin{lem}\label{comp}
Let $C = O_C.H.K$ with $O_C = O_p(C)$. Assume $K$ is a $p$-group and that $H$ is a finite simple group of Lie type in characteristic $p$. Denote $L = O_C.H$, and assume that $L$ is the generalized Fitting subgroup of $C$. Then $\fk C$ has exactly two blocks.
\end{lem}

\begin{proof}
Note that $C$ has one component (a quasisimple group) $E$ with center $Z(E)$ a $p$-group and $E / Z(E) = H$. Elements of the component commute with elements of $O_C$ \cite[31.6(2)]{a00}, so that $C / C_C(O_C)$ is a $p$-group. This implies that there is a bijection between the blocks of $\fk C$ and the blocks of $\fk \tH$ \cite[Thm. 9.10]{n98}.
\end{proof}

\section{The vertices of the indecomposable summands of reduced Lefschetz module}

In this section we determine the vertices of the indecomposable summands of the reduced Lefschetz module associated to the complex of $p$-centric and $p$-radical subgroups, denoted $\Dd(G)$, in a group with parabolic characteristic $p$ which also has certain $p$-local properties. We start by reviewing a few basic facts on Lefschetz modules. All modules are assumed to be finitely generated.

The representation ring or the Green ring $a(G)$ has generators the isomorphism classes $[M]$ of $\fk G$-modules and relations given by direct sums and $\fk$-tensor products. Since the Krull-Schmidt theorem applies, each element of the Green ring can be written as $[M]= \sum n_i [M_i]$, where $[M_i]$ are isomorphism classes of indecomposable $\fk G$-modules. The additive structure of $a(G)$ is of a free abelian group with one generator for each class $[M_i]$.

The Grothendieck ring $R(G)=a(G)/a_0(G)$ is a quotient ring of the Green ring $a(G)$ by the ideal spanned by difference elements of the form $M_2-M_1 -M_3$ where $0 \rightarrow M_1 \rightarrow M_2 \rightarrow M_3 \rightarrow 0$ is a short exact sequence of $\fk G$-modules. Each module is identified with its simple factors from the composition series, so two modules with the same composition factors are identified in the Grothendieck ring. The Grothendieck ring of $G$ over $\fk$ is a free abelian group on the set of isomorphism classes of simple $\fk G$-modules. The simple modules are determined by their (complex) characters and the Grothendieck ring is isomorphic to the ring of characters.

When a group $G$ acts admissibly on a simplicial complex $\Delta$, we can construct the {\it Lefschetz module} by taking the alternating sum of the $\fk$-vector spaces spanned by the chain groups. To obtain the {\it reduced Lefschetz module}, subtract the trivial one dimensional representation. The $n^{\rm th}$ chain group $C_n(\Delta; \fk)$ is a permutation module and a $\fk$-module with oriented simplices as generators. Thus the reduced Lefschetz module is a virtual module, an element of the Green ring:
$$\tL_G(\Delta; \fk) = \sum_{\sigma \in \Delta /G} (-1)^{|\sigma|} \ind _{G_{\sigma}}^G (\fk)- \fk $$

The reduced Lefschetz module associated to a subgroup complex $\Delta$ is not always projective; however, it was shown by Th\'{e}venaz \cite[Thm. 2.1]{th87} that this virtual module turns out to be, in many cases, projective relative to a collection of very small order $p$-groups. Specifically, $\tL_G(\Delta;\fk)$ is $\mathcal{X}$-relatively projective, where $\mathcal{X}$ is a collection of $p$-subgroups such that $\Delta^Q$ is acyclic (for example contractible) for every $p$-subgroup $Q$ which is not in $\mathcal{X}$.

Information about fixed point sets leads to details about the vertices of indecomposable summands of this virtual module. The following theorem is due to Robinson \cite[in proof of Cor. 3.2]{rob88}; also see \cite[Lem. 1]{sa06}. This result is our main tool for finding vertices for reduced Lefschetz modules.

\begin{thm}[Robinson]\label{fixedrob}
The number of indecomposable summands of $\widetilde{L}_G(\Delta;k)$ with vertex $Q$ is equal to the number of indecomposable summands of $\widetilde{L}_{N_G(Q)}(\Delta ^Q;k)$ with vertex $Q$.
\end{thm}

We give below a few applications of Robinson's Theorem \ref{fixedrob}.

\begin{lem}\label{cenvertlef}
Assume $G$ has parabolic characteristic $p$. Suppose that $t$ is an element of order $p$ in $G$ such that $O_p(C_G(t))$ contains central elements. Then no vertex of the reduced Lefschetz module $\tL_G(\Dd(G); \fk)$ contains a conjugate of $t$.
\end{lem}

\begin{proof}
Recall that according to Theorem \ref{ocintilde}, the fixed point set $\Dd(G)^t$ is contractible. Hence $\Dd(G)^t$ is mod-$p$ acyclic, and an application of Smith theory \cite[Thm. VII.10.5(b)]{brown} gives that $\Dd(G)^Q$ is mod-$p$ acyclic for any $p$-group $Q$ containing $t$. Therefore $\tL_{N_G(Q)}(\Dd(G)^Q; \fk) = 0$. Then Theorem \ref{fixedrob} implies that the reduced Lefschetz module $\tL_G(\Dd(G); \fk)$ has no indecomposable summand with vertex $Q$ containing a conjugate of $t$.
\end{proof}

\begin{prop}\label{vertexQ}
Assume $G$ has parabolic characteristic $p$, and let $Q$ be a vertex of an indecomposable summand of the reduced Lefschetz module $\tL_G(\Dd(G); \fk)$. Assume that $t$ is an element of order $p$ in the center $Z(Q)$ such that conditions (N1), (N2) and (N3) hold for $C = C_G(t)$. Then $Q$ is a vertex of a summand of $\tL_C(\Dd(H); \fk)$.
\end{prop}

\begin{proof}
By Robinson's Theorem \ref{fixedrob}, $Q$ is a vertex of a summand of $\tL_{N_G(Q)}(\Dd(G)^Q; \fk)$. Consider the restriction of this Lefschetz module to the subgroup $C \cap N_G(Q)$, obtaining the Lefschetz module $\tL_{N_C(Q)}(\Dd(G)^Q; \fk)$. There exists an indecomposable summand of $\tL_{N_C(Q)}(\Dd(G)^Q; \fk)$ with vertex $Q$ \cite[Lem. 9.5]{alperin}. Since $\Dd(G)^Q = (\Dd(G)^t)^Q$, another application of Theorem \ref{fixedrob} yields a summand of $\tL_C(\Dd(G)^t; \fk)$. Theorem \ref{noncenthm} says that $\Dd(G)^t$ is equivariantly homotopy equivalent to $\Dd(H)$, so $\tL_C(\Dd(G)^t; \fk)$ and $\tL_C(\Dd(H); \fk)$ are isomorphic virtual $\fk C$-modules.
\end{proof}

Note that $t$ can be replaced in the above proof by any subgroup $T \leq Z(Q)$, where $C = TC_G(T)$. Also $C$ can be replaced by $N_G(T)$.

\begin{thm}\label{defT}
Let $G$ be a finite group of parabolic characteristic $p$. Let $T$ be a $p$-subgroup of $G$ and set $C = TC_G(T)$. Assume that the following conditions hold:
\vspace*{-.1cm}
\begin{list}{\upshape\bfseries}
{\setlength{\leftmargin}{1.2cm}
\setlength{\labelwidth}{.8cm}
\setlength{\labelsep}{0.2cm}
\setlength{\parsep}{0cm}
\setlength{\itemsep}{0cm}}
\item[$(i).$] $C = O_C.H.K$ where $O_C = O_p(C)$ and $L = O_C.H$ is the generalized Fitting subgroup of $C$;
\item[$(ii).$] The group $H = L / O_C$ is a finite simple group of Lie type in characteristic $p$;
\item[$(iii).$] There exists a triple of Sylow $p$-subgroups $(S_L,S_C,S)$ with $Z(S) \cap S_L \neq 1$.
\end{list}
Then $T$ is a vertex of an indecomposable summand of the reduced Lefschetz module $\tL_G({\mathcal D}(G); \fk)$ if and only if:
\vspace*{-.2cm}
\begin{list}{\upshape\bfseries}
{\setlength{\leftmargin}{2cm}
\setlength{\labelwidth}{.8cm}
\setlength{\labelsep}{0.2cm}
\setlength{\parsep}{0cm}
\setlength{\itemsep}{0cm}}
\item[$(a).$] $T = O_C$ and $T$ is purely noncentral,
\item[$(b).$] $K = C / L$ is a $p'$-group,
\item[$(c).$] the index of $C$ in $N_G(T)$ is relatively prime to $p$.
\end{list}
Under these conditions, there will exist a unique summand of $\tL_G(\mathcal D(G); \fk)$ with vertex $T$, which will lie in a block with defect group $T$.
\end{thm}

\begin{proof}
Theorem \ref{fixedrob} tells us that $T$ is a vertex of a summand of $\tL_G({\mathcal D}(G); \fk)$ if and only if $T$ is a vertex of a summand of $\tL_{N_G(T)}({\mathcal D}(G)^T;\fk)$. By Lemma \ref{cenvertlef}, if $O_C$ contains central elements $T$ is not a vertex; thus $O_C$ is purely noncentral.

The generalized Fitting subgroup $L$ is characteristic in $C$ and so is normal in $N_G(T)$. Denote the quotient group by $K' = N_G(T) / L$. Note that $K'$ will be a $p'$-group if and only if $K$ is a $p'$-group and the index of $C$ in $N_G(T)$ is relatively prime to $p$. We have $N_G(T) = O_C.H.K'$. The group $C$ has one component $E$ with center $Z(E)$ a p-group and $E / Z(E) = H$. Elements of $O_C$ commute with elements of $E$.

Theorem \ref{noncenthm} tells us that the fixed point set ${\mathcal D}(G)^T$ is $N_G(T)$-homotopy equivalent to the Tits building $\Delta$ for the Lie group $H$. Thus their reduced Lefschetz modules are isomorphic virtual $\fk N_G(T)$-modules. Denote by $M = \tL_{N_G(T)}(\Delta ; \fk)$. The subgroup $O_C$ acts trivially on $M$ (by definition of the action of $N_G(T)$ on $\Delta$) and so $M$ is the inflation of a module $\widetilde{M}$ over the group $\widetilde{H}' = N_G(T) / O_C = H.K'$. The restriction of $\widetilde{M}$ to the Lie group $H$ is up to a sign the Steinberg module ${\rm St}_H$, which is irreducible. This implies that $M$ and $\widetilde{M}$ are also irreducible modules. The restriction $M'$ of $M$ to $L = O_C.H$ is also irreducible, being the inflation of $\pm {\rm St}_H$. The module $\widetilde{M}$ is referred to as an extended Steinberg module; see \cite{schmid92} and \cite{musch95}.

Next, $\fk H$ has only two blocks, the principal block $b_0$ and a defect zero block $b_1$ containing the Steinberg module. Then $\fk L$ has two blocks also, the principal block $B_0$ and a block $B_1$ which contains the module $M'$, using that $E$ centralizes $O_C$ and \cite[Thm. 9.10]{n98}. The block $B_1$ dominates (or contains) the block $b_1$. The normal $p$-subgroup $O_C$ is contained in the vertex $Q_1$ of the irreducible module $M'$ \cite[Thm. 4.7.8]{nt}. The defect group $D_1$ of the block $B_1$ contains $O_C$ (and also $Q_1$), and the quotient $D_1 / O_C$ equals $1$, the defect group of the block $b_1$ \cite[Thm. 9.9(b)]{n98}. Therefore $Q_1 = D_1 = O_C$.

Let $B$ be the block of $\fk N_G(T)$ which contains the module $M$. Note that $B$ covers the block $B_1$. The defect group $D$ of $B$ satisfies $D \cap L = D_1 = O_C$ by \cite[Thm. 15.1(2)]{alperin}. If $Q$ is the vertex of $M$, then $O_C \leq Q$ and $O_C \leq Q \cap L \leq D \cap L = O_C$. Also the quotient $S' = Q / (Q \cap L) = Q / O_C$ is a Sylow $p$-subgroup of $K' = N_G(T) / L$ by \cite[Lem. 9.8]{alperin}. Since $Q = O_C.S'$ is a subgroup of the $p$-group $D$, the quotient map from $N_G(T) \rightarrow K'$ applied to $D$ has kernel $O_C$ and image $S'$. Therefore $Q = D = O_C.S'$.

Since $T \leq O_C$, we have $T = Q$ if and only if $T = O_C$ and $S' = 1$. In this situation, $K'$ is a $p'$-group. Under these conditions, there will be a unique summand of the reduced Lefschetz module $\tL_G(\mathcal D(G); \fk)$ with vertex $T$, and this summand will lie in a block with defect group $T$ because of the relationship between the Brauer correspondence and the Green correspondence.
\end{proof}

\begin{num}\label{mhusch}
Work of M\"{u}hlherr and Schmid \cite[Thm. C and Lem. 7]{musch95} concerning the action of an automorphism of $H$ on the associated building $\Delta$ can be applied to $H \trianglelefteq \widetilde{H}'$. If the value of the rational character $\chi$ of the extended Steinberg module $\widetilde{M}$ evaluated on an element $g \in \widetilde{H}'$ is nonzero, then the fixed point set $\Delta^g$ is a Moufang complex. Also, $\chi(g) \ne 0$ if and only if the $p$-part of $g$ is conjugate to an element of the defect group $S'$ of the block containing $\widetilde{M}$. If $g \in S'$, then $\Delta^g$ is a building for the group $O^{p'}(C_H(g))$. Since $O_C$ acts trivially on $\Delta$, we can phrase this result in terms of the character of the inflated extended Steinberg module $M$, with fixed point set $\Delta^g$, a building when $g \in D = O_C.S'$.

It follows that $\Dd(G)^D$ is homotopy equivalent to the building $\Delta^D=$$\Delta^{S'}$ for $O_C.O^{p'}(C_H(S'))$. As noted in the proof of Lemma \ref{defcp} the extension $H:S'$ is split and $S'$ is a subgroup of $\tH'$. This homotopy is equivariant with respect to $N_G(T) \cap N_G(D)$, which contains the group $DC_G(D)$. To obtain a homotopy equivariant with respect to $N_G(D)$, we require more information about $DC_G(D)$ so that we can apply Theorem \ref{noncenthm} to $DC_G(D)$.
\end{num}

\begin{prop}\label{normalH}
Assume that $H$ is a normal subgroup of $N_G(T)$ with complement a group $D$ containing $O_C = O_p(TC_G(T))$. Denote $S' = D / O_C$, and assume that $D = O_C:S'$ is a split extension. Then $DC_G(D) = DC_H(S')$.
\end{prop}

\begin{proof}
Note that $N_G(T) = H \cdot D = (O_C \times H):S'$. Since $T \leq O_C \leq D$, we have $C_G(D) \leq C_G(T) \leq N_G(T)$, so that $C_G(D) = C_N(D)$, denoting $N = N_G(T)$. An element of $C_N(D)$ can be written uniquely as a product $xy$ with $x \in H$ and $y \in D$, such that for all $z \in D$, $xy = z(xy)z^{-1} = (zxz^{-1})(zyz^{-1})$. Since $zxz^{-1} \in H$ and $zyz^{-1} \in D$, we must have $zxz^{-1} = x$ and so $x \in C_H(D)$. Since elements of $H$ commute with elements of $O_C$, $C_H(D) = C_H(S')$. It now follows that $DC_G(D) = DC_H(D) = DC_H(S')$.
\end{proof}

\begin{num}\label{remD}
In the above situation, we can write $DC_G(D)$ in the form $D.H_D.K_D$, with $H_D = O^{p'}(C_H(S'))$ and $K_D = C_H(S') / H_D$. Usually, when $H$ is a finite simple group of Lie type, the extension $H:S'$ is split and $H_D$ will also be a finite simple group of Lie type. Then $H_D$ will be a component of $DC_G(D)$ and will be normal in $N_G(D)$. As noted in \ref{mhusch}, under the conditions of Theorem \ref{defT}, $\Dd(G)^D$ will be a building and thus not contractible, so $D$ must be purely noncentral. Finally, $K_D$ is a $p'$-group, and the conditions for Theorem \ref{noncenthm} will be satisfied.
\end{num}

\begin{prop}\label{simpleH}
Assume $N = N_G(T) = O_C.(H:S')$ and assume $C_H(S')$ is a simple group. Also assume that $D = O_C.S'$ is purely noncentral in $G$. Then $DC_G(D) = D.C_H(S')$.
\end{prop}

\begin{proof}
Note that $DC_G(D) = DC_N(D) \trianglelefteq N_N(D)$, and under the quotient map $N \rightarrow H:S'$, the image of $N_N(D)$ equals $N_{H:S'}(S') = S'N_H(S')$. Next, take the quotient modulo $S'$ to obtain $N_H(S')$. The image of $DC_G(D)$ under the composition of these two quotient maps is a normal subgroup of the simple group $C_H(S')$. It is impossible for $C_G(D) \leq D$ to be purely noncentral, since if $D \leq S \in \syl_p(G)$ then $Z(S) \leq C_G(D)$. So we must have $DC_G(D) = D.C_H(S')$.
\end{proof}

\begin{prop}\label{O_C}
Let $G$ be a finite group of parabolic characteristic $p$. Let $T$ be a $p$-subgroup of $G$ and set $C = TC_G(T)$. Assume that:
\begin{list}{\upshape\bfseries}
{\setlength{\leftmargin}{1.2cm}
\setlength{\labelwidth}{.8cm}
\setlength{\labelsep}{0.2cm}
\setlength{\parsep}{0.1cm}
\setlength{\itemsep}{0cm}}
\item[$(i).$] Conditions (N1), (N2) and (N3) hold for $C$. In particular $C = O_C.H.K$ with $O_C:=O_p(C)$.
\item[$(ii).$] The subgroup $L =O_C.H$ is the generalized Fitting subgroup of $C$.
\end{list}
Then $O_C \cdot C_G(O_C)$ is of the form $O_C.H.\overline{K}$ for some subgroup $\overline{K} \leq K$, and conditions (N1), (N2) and (N3) are satisfied for $O_C \cdot C_G(O_C)$.
\end{prop}

\begin{proof}
The group $C$ has one component $E$ with center $Z(E)$ a $p$-group and $E/Z(E) = H$. We have $E \leq C_G(O_C)$, and thus $L = O_C.H \leq O_C \cdot C_G(O_C)$. In order to prove that (N2) holds for $O_C \cdot C_G(O_C)$, we have to show that $L$ is a normal subgroup of $N: = N_G(O_C)$. Observe that it suffices to show that $E$ is the only component of $N$, since in this case $L$ is the product of two normal subgroups of $N$, namely $O_C$ and $E$. Let $E(N)$ denote the layer of $N$, the central product of all the components of $N$. Since $E(N)$ centralizes the Fitting subgroup of $N$, which contains $O_C$, it follows that $E(N) \le O_C \cdot C_G(O_C) \le C$.  Thus $E(N) = E$. Define $\overline{K}$ to be the image of $O_C \cdot C_G(O_C)$ under the quotient map $C \rightarrow C/L = K$. Thus conditions (N1) and (N2) are satisfied for $O_C \cdot C_G(O_C)=O_C.H.\overline{K}$. Condition (N3) also holds since any Sylow $p$-subgroup $S_{O_C}$ of $O_C \cdot C_G(O_C)$ can be extended to a Sylow $p$-subgroup $S_C$ of $TC_G(T)$, and we have $S_L = S_C \cap L$ is a Sylow $p$-subgroup of $L$. Then using \ref{n3all} there exists a Sylow $p$-subgroup $S$ of $G$ extending $S_C$ such that $Z(S) \cap S_L \neq 1$.
\end{proof}

The reduced Lefschetz module $\tL_G(\Cc (G); \fk)$ was studied by Sawabe \cite{sa06}, in the more general context when $G$ is any finite group and $\Cc (G)$ is a collection of subgroups of $G$ which is closed under $p$-overgroups. Sawabe applied his results to the collection $\Cc (G)$ of $p$-centric subgroups of $G$, which is equivariantly homotopy equivalent to $\Dd (G)$. The collection $\tSs(G)$ also satisfies this condition by \ref{tildeclosed}(i), and for $G$ of parabolic characteristic $p$ is equivariantly homotopy equivalent to $\Dd(G)$ by Theorem \ref{propdist}(iii-iv). We give below two of Sawabe's results.

\begin{prop}[Sawabe]\label{swbe}
Assume that $G$ is a finite group and let $\Cc (G)$ denote a collection of $p$-subgroups of $G$ that is closed under taking $p$-overgroups. Take $V$ in $\Bb(G) \setminus \Cc(G)$ of maximal order. Then the following hold:
\begin{list}{\upshape\bfseries}
{\setlength{\leftmargin}{.8cm}
\setlength{\labelwidth}{1cm}
\setlength{\labelsep}{0.2cm}
\setlength{\parsep}{0cm}
\setlength{\itemsep}{0cm}}
\item[$(i).$]\cite[Prop. 4]{sa06} For any $p$-subgroup $Q$ of order greater than that of $V$, the fixed point set $\Cc(G)^Q$ is $N_G(Q)$-contractible. In particular $Q$ is not a vertex of $\tL_G(\Cc(G); \fk)$.
\item[$(ii).$]\cite{sawp} The group $V$ is a vertex of an indecomposable summand of $\tL_G(\Cc(G); \fk)$ if and only if $\tL _{N_G(V)} (\Cc(G)^V; \fk) \ne 0$.
\end{list}
\end{prop}

We end this Section with a result which asserts that under our working assumptions the defect groups of the nonprincipal blocks of $\fk G$ are vertices for the nonprojective indecomposable summands of the reduced Lefschetz module associated to the collection of $p$-centric and $p$-radical subgroups of $G$.

\begin{thm}\label{noncentricthm}
Let $G$ be a finite group of parabolic characteristic $p$. Let $T$ be a $p$-subgroup of $G$ and set $C = TC_G(T)$. Assume the following hold:
\begin{list}{\upshape\bfseries}
{\setlength{\leftmargin}{1.2cm}
\setlength{\labelwidth}{.8cm}
\setlength{\labelsep}{0.2cm}
\setlength{\parsep}{0.1cm}
\setlength{\itemsep}{0cm}}
\item[$(i).$] Conditions (N1), (N2) and (N3) hold for $C$. Thus $C = O_C.H.K$ and $N_G(T) = O_C.H.K'$. \item[$(ii).$] $H$ is a finite group of Lie type in characteristic $p$.
\item[$(ii).$] The defect group $D$ of the block $B$ for the group $N_G(T)$ which contains the inflated extended Steinberg module of $H$ is a noncentric $p$-radical subgroup of $G$ of maximal order.
\end{list}
Then $D$ is a vertex of an indecomposable summand of the reduced Lefschetz module $\tL_G(\Dd(G); \fk)$, and if this summand lies in a nonprincipal block, the defect group of this block equals $D$.
\end{thm}

\begin{proof}
Theorem \ref{noncenthm} tells us that $\Dd(G)^T$ is equivariantly homotopy equivalent to the building $\Delta$ for the Lie group $H$. Thus $\Dd(G)^D$ is homotopy equivalent to $\Delta^D$, a building by \cite{musch95}, see also \ref{mhusch}. This latter homotopy is only equivariant with respect to $N_G(T) \cap N_G(D)$, but all we need is an ordinary homotopy to conclude that the reduced Lefschetz module $\tL_{N_G(D)}(\Dd(G)^D; \fk)$ is nonzero. Thus there is at least one indecomposable summand of $\tL_G(\Dd(G); \fk)$ with vertex $D$ by Sawabe's Proposition  \ref{swbe}(ii). The defect group $D_0$ of the block containing this summand must contain $D$, and if this block is not the principal block, the defect group $D_0$ is purely noncentral by Lemma \ref{blpar}. If $D \lneq D_0$ then $D_0$ is centric (note $D_0$ is $p$-radical, and $D$ was assumed to be a maximal noncentric $p$-radical subgroup). But for any Sylow $p$-subgroup $S$ of $G$ with $D_0 \leq S$, we have $Z(S) \leq C_G(D_0)$ and thus $Z(S) \leq D_0$ since $Z(D_0)$ is the unique Sylow subgroup of $C_G(D_0)$. This contradicts $D_0$ being purely noncentral. Thus $D = D_0$.
\end{proof}

With additional information, including $DC_G(D) = D.C_H(S')$, as in Propositions \ref{normalH} and \ref{simpleH}, we would be able to obtain a homotopy equivariance with respect to $N_G(D)$, so that $\tL_{N_G(D)}(\Dd(G)^D; \fk)$ would be the irreducible inflated extended Steinberg module for the Lie group $O^{p'}(C_H(S'))$. Then there would be exactly one indecomposable summand of $\tL_G(\Dd(G); \fk)$ with vertex $D$.

\section{Applications: sporadic simple groups of parabolic characteristic p}

In this final Section we consider a few examples. We discuss how our results apply to the sporadic simple groups of parabolic characteristic $p$. We start with even characteristic, after which we discuss examples in characteristic $3$.

\subsubsection*{{\bf Sporadic simple groups of parabolic characteristic $\mathbf{2}$}}
Among the $26$ sporadic groups nine have local characteristic $2$; these are: $M_{11}$, $M_{22}$, $M_{23}$, $M_{24}$, $J_1$, $J_3$, $J_4$, $Co_2$ and $Th$. The twelve groups given in Table $6.1$ have parabolic characteristic $2$ and contain noncentral involutions whose centralizers have component type. The remaining five groups have central involutions with centralizer of component type; these are $Co_3$, $Fi_{23}$, $McL$, $O'N$ and $Ly$.

In Table $6.1$ below we gather information regarding the sporadic simple groups of parabolic characteristic $2$. The notation is as follows. We let $t$ be a noncentral involution in $G$. In the second column, the centralizer $C_G(t)$ is given. $H$ denotes a group associated to the component (a central quotient or an extension of the component) of $C_G(t)$, usually a group of Lie type described in the Atlas \cite{atlas} notation. We are studying $\Dd_2(G)$, the collection of $2$-centric and $2$-radical subgroups, or equivalently the distinguished $2$-radical subgroups, as defined in \ref{hatdef}. In the fourth column we describe the fixed point set for $T=\langle t \rangle$; this is usually a building, and our notation follows \cite[p. 214]{ronanbook}. For the last three groups, $\Dd_2(G)^T$ is not a building, but it is homotopy equivalent to $\Dd_2(H)$ and $H$ also has parabolic characteristic $2$. Columns $5$, $6$ and $7$ contain information on a vertex $V$ of a non-projective indecomposable summand of $\tL_{C_G(t)}(\Dd_2(G)^t; \mathbb{F}_2)$. Finally, $H_V$ denotes the Lie type component of $VC_G(V)$ and the column $\Dd_2(G)^V$ gives the type of the corresponding building associated to this fixed point set under the action of $V$.

The main sources of information are the Atlas \cite{atlas}, Chapter $5$ in the third volume of the classification of the finite simple groups \cite{gls3} and Landrock's paper \cite{lan78} on the non-principal $2$-blocks of the sporadic simple groups. Additional references are provided in the foremost right column of the table.

The following Proposition is a direct consequence of our results from Theorems \ref{noncenthm} and \ref{defT} combined with the information from Table $6.1$.

\begin{prop}\label{par2}
Let $G$ be a sporadic simple group of parabolic characteristic $2$ and let $\tL_G$ denote the reduced Lefschetz module for the complex of $2$-radical and $2$-centric subgroups in $G$, over $\mathbb{F}_2$.
\begin{list}{\upshape\bfseries}
{\setlength{\leftmargin}{.8cm}
\setlength{\labelwidth}{1cm}
\setlength{\labelsep}{0.2cm}
\setlength{\parsep}{0.5ex plus 0.2ex minus 0.1ex}
\setlength{\itemsep}{.6ex plus 0.2ex minus 0ex}}
\item[$(i).$] Let $G$ be one of the groups $M_{12}$, $J_2$, $HS$, $Ru$, $Suz$, $Fi_{22}$, $He$, $Co_1$ or $BM$. The virtual module $\tL_G$ has precisely one indecomposable non-projective summand $M$, and this summand lies in a non-principal $2$-block of $\mathbb{F}_2G$ and has vertex $V$ equal to the defect group of this block. The module $M$ is the Green correspondent of the inflated extended Steinberg module for the building given in column $\Dd_2(G)^V$ of Table $6.1$.
\item[$(ii).$] If $G$ is one of the groups $Fi'_{24}, HN$ or $M$ then the reduced Lefschetz module $\tL_G$ has at least one indecomposable non-projective summand, which lies in a non-principal block of $\mathbb{F}_2G$ and has vertex $V$ equal to the defect group of this block.
\end{list}
\end{prop}

\begin{num}
In every case, the vertex $V$ referred to in the Proposition is the unique maximal noncentric $2$-radical subgroup of $G$, as suggested by the work of Sawabe, see Proposition \ref{swbe}(i). For example, when $G$ is either $HN$ or $M$, there are five conjugacy classes of noncentric $2$-radical subgroups, with the semidihedral group $SD_{16}$ maximal; see \cite[Table VI]{y02} for $HN$ and \cite[Table 2]{y05a} for $M$. In both of these cases\footnote{Recall the notation described in Introduction.}, there are inclusions $2A \leq 2A^2 \leq D_8 \leq SD_{16}$ and also $2A \leq Q_8 \leq SD_{16}$. For $G = BM$, there are four conjugacy classes of noncentric $2$-radical subgroups, with inclusions $2A \leq 2A_2C_1 \leq D_8$ and $2C^3 \leq D_8$; see \cite[Table 1]{y05a}. For $G = Fi'_{24}$, there are four classes of noncentric $2$-radical subgroups, including two classes of elementary abelian subgroups $2A^2$ of rank $2$, with inclusions $2A \leq 2A^2_I \leq D_8$ and $2A \leq 2A^2_{II} \leq D_8$; see \cite[Table 2]{ky}. Regarding the remaining groups from Table $6.1$: $J_2$, $Ru$ and $Fi_{22}$ have one conjugacy class of noncentric $2$-radical subgroups each; $M_{12}$ and $HS$ have two classes of such subgroups each and $Co_1$, $Suz$ and $He$ have three classes each. It is easy to check that the poset of noncentric $2$-radical subgroups has a unique maximal member in each case. Although not all the details are given, the relevant references for each sporadic simple group can be found in Chapter $7$ of the book by Benson and Smith \cite{bs04}.
\end{num}

\begin{num}
The three cases in Table $6.1$ in which the fixed point set is contractible are verified using Theorem \ref{ocintilde} and Lemma \ref{cenvertlef} since $O_C$ contains central elements. The following information is available in the Atlas \cite{atlas}. For $t = 2C$ in $G = Fi_{22}$, there is a purely central elementary abelian $2$-group $2B^4$ of rank $4$ in the subgroup $2^5 = 2A_6B_{15}C_{10}$ of $2^{5+8}$. For $t = 2C$ in $G = Co_1$, the elementary abelian $2^{11} = 2A_{759}C_{1288}$ contains central elements of class $2A$. The normalizer is $N(2^{11}) = 2^{11}:M_{24}$, and the subgroup $M_{12}.2 \leq M_{24}$ has index $1288$. And for $t = 2D$ in $G = BM$, the elementary abelian subgroup $2^9$ of $2^9.2^{16}$ contains a purely central $2B^8$ of rank $8$. The normalizer is $N(2B^8) = 2^9.2^{16}.Sp_8(2)$. The lift of the group $2^9$ in $2.BM \leq M$ is a group $2^{10}$ referred to as an ``ark" in the work of Meierfrankenfeld and Shpectorov \cite{meierspec02,meirmon03}. The notation $O_8^+(2)$ in $C(2D) = 2^9.2^{16}.O_8^+(2).2$ is the Atlas notation for the simple group, also denoted $\Omega_8^+(2)$, so that $O_8^+(2).2 = SO_8^+(2)$.
\end{num}

\begin{num}
The next step in the proof of the Proposition \ref{par2} is to verify that the hypotheses of Theorem \ref{noncenthm} hold for $C_G(t)$ and for $VC_G(V)$. In particular, we need to check that condition $({\rm N3})$ holds in these groups. We first recall this condition. Let $T$ be a $p$-subgroup of $G$. If $TC_G(T)=C = O_C.H.K = L.K$, then $({\rm N3})$ asserts that there exists a triple of Sylow $p$-subgroups $S_L \trianglelefteq S_C \le S$ of $L, C$ and $G$ respectively, such that $Z(S) \cap S_L \ne 1$; in particular the group $S_L$ is distinguished in $G$. Note this condition is trivially satisfied when $K$ is a group of odd order, since then $S_L = S_C$. This happens in the following four cases: when $G$ equals $J_2$, $Ru$ or $M$, and for $t = 2A$ in $G = Fi_{22}$. It also occurs when $C = VC_G(V)$ for every group $V$ from the fifth column of Table $6.1$.

We now prove a Lemma which yields condition $({\rm N3})$ in five additional cases from Table $6.1$.

\begin{lem}
Let $G$ be one of the following groups: $M_{12}$, $HS$, $Suz$, $He$, or $BM$. Let $t$ be a noncentral involution in $G$ (where for $G = BM$ we assume $t$ is of class $2A$), and let $L$ be the component of the centralizer $C_G(t)$. Let $\overline{L} = L$ except in the case $G = HS$, where $\overline{L} = S_6 \leq Aut(A_6)$. Then the $2$-central involutions of $G$ which lie in $C_G(t)$ actually lie in $\overline{L}$.
\end{lem}

\begin{proof}
The first three groups (and also $J_2$ and $Ru$) are discussed in \cite[Lemma 2.3]{mgo5}. For $G = HS$, this result is Lemma $1.7$ of Aschbacher \cite[p. 24]{a03}. For $G=Suz$, see the paragraph on p. $456$ before Lemma $1$ of Yoshiara \cite{y01}. For $G = M_{12}$, the centralizer $C_G(t) = 2 \times S_5$ has five conjugacy classes of involutions, fusing to two classes in $M_{12}$. The involutions in $A_5$ are squares of elements of order $4$, which are $2$-central in $M_{12}$ by the Atlas \cite{atlas}. The other four classes in $C_G(t)$ are noncentral in $M_{12}$.  There exists a purely noncentral elementary abelian $2^2$ in $M_{12}$, so the involutions in $S_5 \setminus A_5$ are noncentral.  The $2$-central involutions are closed under nontrivial commuting products; the class multiplication coefficient is zero \cite[p. 107, 162]{bs04}. Thus the diagonal involutions in $2 \times A_5$ are noncentral in $M_{12}$. Next, for $G = He$, the Lemma follows from \cite[p.277--278]{held69}. Finally, for $t = 2A$ in $G = BM$, the result follows from Segev \cite[2.5.2 and 3.18]{segev91}.
\end{proof}

Note that if $L = C_G(O_C)$ then whenever $S_L \leq S$, we have $Z(S) \leq L$. This argument shows that condition $({\rm N3})$ is satisfied for $G$ equal to $Suz$, $He$, also for $t=2A$ in $Fi_{22}$, $t=2B$ in $Co_1$ and for $t = 2C$ in $G = BM$. Condition $({\rm N3})$ holds for $G = HN$ by a description of the groups $S_L$ and $S_C$ given in \cite[(L) on p. 128]{harada75}, where it is shown that $Z(S_C) = 2^2 \leq S_L$.

Finally, let $G = Fi_{24}'$, with $C_G(t) = 2^.Fi_{22}:2$ and $L = 2^.Fi_{22}$. Let $S$ denote a Sylow $2$-subgroup of $Fi_{22}$, with $S_L = 2^.S$ and $S_C = 2^.S:2$. Let $z$ be an involution in the center $Z(S_C)$ and assume that $z$ is not in $S_L$.  Consider the centralizer $C_G(2^2)$ of the elementary abelian $2^2 = \langle t, z \rangle$, which must contain $S_C$.  But by \cite[Table 10]{wi87}, this centralizer is divisible by only $2^{14}$ (it equals either $2^2 \times O_8^+(2):3$ or $2^2 \times 2^6:U_4(2)$).  This contradiction implies that the subgroup $\Omega_1(Z(S_C))$ containing all involutions of the center $Z(S_C)$ must be a subgroup of $S_L$. This implies that condition $({\rm N3})$ holds.
\end{num}

\begin{num}
In what follows we determine those $2$-subgroups which are candidates for the vertices of $\tL_G$. Note that Proposition \ref{vertexQ} implies that we need only to consider those $2$-subgroups $Q$ which are vertices of the reduced Lefschetz module $\tL_{C_G(t)}(\Dd_2(G)^t; \mathbb{F}_2)$ for the action of the centralizer $C_G(t)$, with $t$ an involution in $Z(Q)$, on the complex $\Dd_2(H)$ for $H$ described in the fourth column of Table $6.1$. This yields those groups given in column five, denoted $V$.

Let us consider the first nine groups in Table $6.1$. Theorem \ref{defT} applies for both $C = C_G(t)$ and for $C = VC_G(V)$ in every case.  Observe that for $G = BM$, the group $V = 2^2 = 2A_1C_2$ is not a vertex of a summand of the reduced Lefschetz module for $G$ since the centralizer of $V$ has index $2$ in the normalizer of $V$.

In the remaining three cases, Theorem \ref{defT} does not apply if we let $C = C_G(t)$. The subgroup $C_G(t)$ is of the form $O_C.H.K$ but $H$ is not a group of Lie type. For each of these groups, Theorem \ref{noncenthm} tells us that the fixed point set $\Dd_2(G)$ is equivariantly homotopy equivalent to the complex of $2$-centric and $2$-radical subgroups in the component $H$ of $C_G(t)$. We can, however, apply Theorem \ref{defT} to the group $C = VC_G(V)$ from the sixth column of the Table. Specifically, $V = D_8$ in $G = Fi'_{24}$ and $V = SD_{16}$ in $G = HN$ and $G = M$, obtaining the existence of summands with vertex equal to the dihedral or semidihedral group.
For $G = M$ it is possible that $T = 2A$ is a vertex of a summand of the reduced Lefschetz module (if $BM$ has a projective summand in its reduced Lefschetz module). For $G = HN$, we do not know the vertices for the action of $2.HS:2$ on the complex $\mathcal D_2(HS)$. Finally, we remark that for $G = Fi'_{24}$ and for $G = HN$ the group $V = 2^2$ is not a vertex, since the centralizer of $V$ has even index in the normalizer of $V$.
\end{num}

\subsubsection*{{\bf Our results and the Benson-Grizzard-Smith conjecture}}

The book of Benson-Smith \cite{bs04} describes a $2$-local geometry for each of the $26$ sporadic finite simple groups. On pages $105-108$, they list those sporadic groups for which the $2$-local geometry is equivariantly homotopy equivalent to the Brown complex\footnote{These are: $M_{11}$, $M_{22}$, $M_{23}$, $M_{24}$, $J_1$, $J_3$, $J_4$, $McL$, $Co_2$, $Th$, $Ly$.} $\Ss_2(G)$ \cite[(5.9.2)]{bs04}, or to the Benson complex\footnote{Corresponds to a certain collection of elementary abelian $2$-groups, see \cite[(5.10.1)]{bs04} for a precise definition.}$ ^,$\footnote{These are: $M_{12}$, $J_2$, $HS$, $Suz$, $Co_3$, $Ru$.} $\mathcal{E}_2(G)$ \cite[(5.10.3)]{bs04}, or to the complex\footnote{These are: $M_{12}$, $J_2$, $Suz$, $Co_1$, $Fi_{22}$, $Fi_{23}$, $Fi_{24}'$, $HN$, $B$, $M$, $He$, $Ru$, $O'N$.} $\Dd_2(G) = B_2^{cen}(G)$ \cite[(5.11.3)]{bs04}. The eleven groups in list \cite[(5.9.2)]{bs04} have projective reduced Lefschetz modules; all but one case\footnote{The group $Ly$.} appear in \cite{rsy} and \cite{sy}. The six groups in list \cite[(5.10.3)]{bs04} are studied by Grizzard in his thesis \cite{grithesis}. Grizzard computes reduced Lefschetz characters for these six groups, and also for the groups $O'N$ and $He$. The characters were already known for three of these eight groups, namely $M_{12}$ \cite[p. 44]{bw93}, $J_2$ \cite[after Thm. 7.7.1]{bs04} and $HS$ (Klaus Lux, unpublished).

Four of the groups in the list \cite[(5.10.3)]{bs04} also appear in \cite[(5.11.3)]{bs04}, and it seems that $HS$ should also appear in both lists. Note that although $HS$ does not appear in \cite[(5.11.3)]{bs04}, in \cite[p. 193]{bs04} a quotient map is described yielding a homotopy equivalence of the $2$-local geometry of $HS$ with $\Dd_2(HS)$. The sporadic groups of parabolic characteristic $2$ given in Table $6.1$ consist of eleven of the thirteen groups in list \cite[(5.11.3)]{bs04}, plus $HS$. Thus Proposition $6.1$ applies to the $2$-local geometries for these twelve groups. The remaining two groups in \cite[(5.11.3)]{bs04} are $O'N$ and $Fi_{23}$, and in both groups all involutions are central. Then our distinguished $2$-radical complex equals the Bouc complex of all nontrivial $2$-radical subgroups, which is homotopy equivalent to the Brown complex of all nontrivial $2$-subgroups. Thus the results in this paper do not apply to the $2$-local geometries for the two groups $O'N$ and $Fi_{23}$.

In a footnote in \cite[p. 164]{bs04}, the following comment appears: ``For all $15$ sporadic groups $G$ in (5.10.3) and (5.11.3) for which the $2$-local geometry is not homotopy equivalent to $\Ss_2(G)$, it seems that the reduced Lefschetz module involves an indecomposable in a suitable non-principal $2$-block of $G$ of small positive defect". Grizzard refers to this as the ``Nonprincipal Block Observation" \cite[Rem. 3.1]{gri} and conjectures that each sporadic group affording a nonprojective Lefschetz module will have a nonprojective part in a nonprincipal block. Grizzard verifies this conjecture for the eight groups in his thesis. Grizzard also shows that for seven of his groups, the defect group of this nonprincipal block has order equal to the ratio of the $2$-part $|G|_2$ of the order of the group and the $2$-part of the reduced Euler characteristic of the $2$-local geometry. This was also predicted by Steve Smith; see for example \cite{sds05}. Grizzard observed that this was not true for $O'N$, which also differed from the other examples by having a nonprojective part of the reduced Lefschetz module lying in the principal block.

Sawabe approached this latter idea from the viewpoint of vertices of modules, and proves a theorem concerning a lower bound for the $p$-part of the reduced Euler characteristic of the complex $\Dd_p(G)$ \cite[Prop. 7]{sa06} in terms of the order of a noncentric $p$-radical subgroup $V$ of maximal order. An unpublished result of Sawabe \cite{sawp} also shows that such a maximal noncentric $p$-radical subgroup $V$ is a vertex of a summand of the reduced Lefschetz module $\tL_G(\Dd(G); \fk)$ if and only if $\tL_W(\Ss(W); \fk) \neq 0$, for $W = N_G(V)/V$. This result can be applied to many groups, including all of those in Table $6.1$, showing the existence of a summand of the reduced Lefschetz module with vertex $V$. This result also applies to the sporadic group $Fi_{23}$ and the maximal noncentric $2$-radical subgroup $V = D_8$ (see \cite{anb99} for the radical subgroups of $Fi_{23}$, and for the normalizer $N_G(V) = D_8 \times Sp_6(2)$). It follows that the dihedral group $D_8$ is a vertex of a summand of the reduced Lefschetz module for the $2$-local geometry of $Fi_{23}$. But the work of Sawabe does not give any information on the block in which this summand lies. Also, his results yield existence of nonprojective summands for all the groups in Table $6.1$, but does not supply the uniqueness of the nonprojective summand that we obtain in Proposition $6.1(i)$. Indeed, it has been shown that there are examples where there are more than one nonprojective summand. Grizzard shows that the reduced Lefschetz module for the $2$-local geometry of the sporadic group $O'N$ has at least two nonprojective summands, one in the principal block and one in a block of defect three. And for $Co_3$, there are precisely three nonprojective summands, all lying in a nonprincipal block of defect three; see \cite{mgo5}.

\subsection*{Sporadic simple groups of parabolic characteristic $\mathbf{p=3}$}

We now discuss the cases when $G$ is a sporadic simple group of parabolic characteristic $3$. The notations from Table $6.1$ are maintained in Table $6.2$ and below. In what follows $t$ will denote an element of order $3$ and $O_C = O_3(C_G(\langle t \rangle))$.

\begin{prop}
Let $G$ be a sporadic simple group of parabolic characteristic $3$ and let $\tL_G$ denote the reduced Lefschetz module for the complex of $3$-radical and $3$-centric subgroups in $G$, over $\mathbb{F}_3$.
\begin{list}{\upshape\bfseries}
{\setlength{\leftmargin}{.8cm}
\setlength{\labelwidth}{1cm}
\setlength{\labelsep}{0.2cm}
\setlength{\parsep}{0cm}
\setlength{\itemsep}{0cm}}
\item[$(i).$] Let $G$ be one of the groups $M_{12}$, $Co_3$, $J_3$, $Co_2$, $Fi_{22}$, $Fi_{23}$, $Fi'_{24}$ or $Th$. The virtual module $\tL_G$ has precisely one indecomposable non-projective summand $M$, and this summand lies in a non-principal $3$-block of $\mathbb{F}_3G$ and has vertex equal to the defect group of this block. The module $M$ is the Green correspondent of the inflated extended Steinberg module for the building given in the eighth column of Table $6.2$.
\item[$(ii).$] If $G$ is one of the groups $HN, BM$ or $M$ then the reduced Lefschetz module $\tL_G$ has at least one non-projective indecomposable summand which lies in a non-principal block of $\mathbb{F}_3G$.
\end{list}
\end{prop}

\begin{proof}
The eight cases in Table $6.2$ in which the fixed point set is contractible are verified using Theorem \ref{ocintilde} and Lemma \ref{cenvertlef} since $O_C$ contains central elements. For two of these cases, the necessary information in is the Atlas \cite{atlas}. Let $t = 3B$ in $G = Co_3$. The elementary abelian $3^5$ is of type $3A_{55}B_{66}$, containing central elements of type $3A$. Next, let $t = 3C$ in $G = Fi'_{24}$. The elementary abelian $3^7$ is of type $3A_{378}B_{364}C_{351}$, containing central elements of type $3B$. The normalizer of $3^7$ is $N(3^7) = 3^7.O_7(3)$ and the subgroup $2.U_4(3):2 \leq O_7(3)$ has index $351$. The group $3^7$ is the natural orthogonal module for $O_7(3)$ with the $364$ subgroups of type $3B$ corresponding to the totally isotropic points \cite[Sect. 4.1]{ky}.

We now quote information from three papers of R. A. Wilson. For $t = 3C$ in $G = Th$ the elementary abelian $3 \times 3^4$ is of type $3B_{40}C_{81}$ with type $3B$ central; see \cite[p.21]{wil88b}.

For $t = 3C$ in $G = Fi_{22}$ the elementary abelian $3^5$ is of type $3A_{36}B_{40}C_{45}$, containing central elements of type $3B$; see \cite[p. 201]{wi84}. Note that $N(3^5) = 3^5.U_4(2).2$ and the subgroup $2.(A_4 \times A_4).2 \leq U_4(2)$ has index $45$.

For $t = 3C$ in $G = Fi_{23}$ the elementary abelian $3^6$, whose normalizer is $3^6.L_4(3).2$, contains elements of type $3A, 3B$ and $3C$, with $3B$ being $3$-central. This follows from the fact that the group $3^6$ has an invariant quadratic form defined by putting the norms of elements of classes $3A$, $3B$ and $3C$ equal to $-1$, $0$ and $1$ respectively; see \cite[Prop. 1.3.3]{wi87}. This indirectly states that $3$-central elements $3B$ occur in $3^6$.

For $G = Fi'_{24}$, with type $3B$ central, see \cite[p. 87]{wi87} for a discussion of the normalizer of a purely central elementary abelian $3B^2$ of rank $2$. The normalizer is $N(3B^2) = 3^2.3^4.3^8:(A_5 \times 2.A_4(2).2)$. The group $3^2.3^4 = 3^6$ has an $A_5$ action on the $3^4$-factor with orbits of size $5$, $10$, $10$, $15$ on the $40$ cyclic subgroups. The corresponding cosets of $3^2 \leq 3^2.3^4$ consist of elements of type $3D$, $3A$, $3B$ and $3C$ respectively. This yields a subgroup $C(3D) = 3^2.3^4.3^6.(A_4 \times 2.A_4) \leq N(3B^2)$. Thus the $3^2$ in $C(3D)$ is the purely central $3B^2$.

Lastly, let $G$ be either $Fi_{22}$ or $Fi_{23}$ and let $t = 3D$. The fact that $O_C$ contains central elements of type $3B$ was not immediately evident from the literature, and we would like to thank Ronald Solomon for providing the following argument. Let $N = N_G(\langle t \rangle)$, and denote $Q = O_3(N)$, which is equal to $O_C$. Then the quotient group $N/Q$ is either $2.S_4 = GL_2(3)$ if $G = Fi_{22}$, or $2 \times 2.S_4$ if $G = Fi_{23}$. Let $P \in Syl_3(N)$ be a Sylow $3$-subgroup of the normalizer, and extend it to a Sylow $3$-subgroup $S$ of $G$, so that $P \leq S \in Syl_3(G)$. Note that $Q \leq P$ is a subgroup of index three. Since $t \in Q \leq P \leq S$, we have $Z(S) \leq C_G(t) \leq N$ and also $Z(S) \leq C_G(Q)$, so that $Z(S) \leq C_N(Q)$. We want to show that $Z(S) \leq Q$. Assume that this is not true, so that $P = Q \cdot Z(S)$ and also under the quotient map $q : N \rightarrow N/Q$, the image $q(Z(S))$ is a Sylow $3$-subgroup of $GL_2(3)$. Note that for any $g \in N$, the conjugate $gZ(S)g^{-1} \leq C_N(Q)$, and the images of these conjugates in the quotient group $N/Q$ are the Borel subgroups of $GL_2(3)$ and thus generate $2.A_4 = SL_2(3)$. Therefore there is an involution $z \in C_N(Q)$ which maps to the central involution of the matrix group in $N/Q$. Clearly $Q \leq C_G(z)$, and it also follows that $Z(S) \leq C_G(z)$. Thus $P \leq C_G(z)$ and the order of $P$ divides the order of $C_G(z)$, with $|P| = 3^7$ if $G = Fi_{22}$ and $|P| = 3^{10}$ if $G = Fi_{23}$. Each of these Fischer groups has three conjugacy classes $2A$, $2B$ and $2C$ of involutions. In $Fi_{22}$ their centralizers have orders $2^{16} \cdot 3^6 \cdot 5 \cdot 7 \cdot 11$, $2^{17} \cdot 3^4 \cdot 5$ and $2^{16}\cdot  3^3$ respectively. In $Fi_{23}$, the centralizers of involutions have orders $2^{18} \cdot 3^9 \cdot 5^2 \cdot 7 \cdot 11 \cdot 13$, $2^{18} \cdot 3^6 \cdot 5 \cdot 7 \cdot 11$ and $2^{18} \cdot 3^5 \cdot 5$. This is a contradiction to $P \leq C_G(z)$, and therefore $Q = O_C$ contains central elements of type $3B$.

The condition (N3) concerning the triple of Sylow subgroups is trivially satisfied in every case except one, namely $C(3A)$ in $Fi'_{24}$. In this case the stronger condition is satisfied: every central element of type $3B$ which lies in $C(3A) = 3 \times O_8^+(3):3$ actually lies in $3 \times O_8^+(3)$. This follows from the discussion in \cite[p. 83]{wi87} about a symplectic form defined using commutators and lifts to the group $3.Fi'_{24}$ in $M$. The lift of $3 \times O_8^+(3):3$ is $O_8^+(3):3^{1+2}$.

Theorems \ref{noncenthm} and \ref{defT} now imply the results of this Proposition. Note that for $G = HN, BM$ or $M$ the reduced Lefschetz module may have other nonprojective summands with vertex a cyclic group of order three.
\end{proof}

\newpage
\begin{center}
{\scriptsize
\begin{tabular}{|c|l|c|l|c|l|l|c|c|l|}
\hline
1&2&3 &4 &5 &6 &7 &8 &9 &10 \\
\hline
    & & & & & & & & & \\
$G$ & $C_G(t)$ &$H$& $\Dd_2(G)^t$&$V$&$VC_G(V)$&$N_G(V)$&$H_V$ & $\Dd_2(G)^V$ & Ref. \\
    & & & & & & & & & \\
\hline
\hline
& & & & & & & & & \\
$M_{12}$ & $C(2A)=2 \times A_5.2$ & $L_2(4)$ & $A_1$ & $2^2$ & $2^2 \times S_3$ &$A_4 \times S_3$&$L_2(2)$ & $A_1$&\cite{mgo5}\\
& & & & & & & & & \\
\hline
& & & & & & & & & \\
$J_2$ & $C(2B)=2^2 \times A_5$ &$L_2(4)$ & $A_1$ & $2^2$ & $2^2 \times A_5$ & $A_4 \times A_5$ &$L_2(4)$& $A_1$&\cite{mgo5}\\
& & & & & & & & & \\
\hline
& & & & & & & & & \\
$HS$ & $C(2B)=2 \times A_6.2^2$ &$Sp_4(2)$& $C_2$ & $2^2$&$2^2 \times 5:4$ & $2^2 \times 5:4$ & $Sz(2)$&$A_1$&\cite{mgo5}\\
& & & & & & & & & \\
\hline
& & & & & & & & & \\
$Ru$ & $C(2B)=2^2 \times Sz(8)$ &$Sz(8)$& $A_1$ & $2^2$&$2^2 \times Sz(8)$ &$(2^2 \times Sz(8)):3$ & $Sz(8)$&$A_1$&\cite{mgo5}\\
& & & & & & & & & \\
\hline
& & & & & & & & & \\
$Suz$ & $C(2B)=(2^2 \times L_3(4)):2$ &$L_3(4)$& $A_2$ & $D_8$&$D_8 \times 3^2:Q_8$& $D_8 \times 3^2:Q_8$ & $U_3(2)$&$C_1$&\cite{mgo5}\\
& & & & & & & & & \\
\hline
& & & & & & & & & \\
$Fi_{22}$ & $C(2A)=2^.U_6(2)$ &$U_6(2)$& $C_3$ & $2$&$2^.U_6(2)$& $2^.U_6(2)$ & $U_6(2)$&$C_3$& \cite{mgo4}\\
          & $C(2C)=2^{5+8}:(S_3 \times 3^2:4)$ & & contr. & && & & &\\
& & & & & & & & & \\
\hline
& & & & & & & & & \\
$He$ & $C(2A)=(2^2\;^. L_3(4)):2$ &$L_3(4)$& $A_2$ & $D_8$&$D_8 \times L_3(2)$ & $D_8 \times L_3(2)$ & $L_3(2)$& $A_2$& \cite{held69}\\
& & & & & & & & & \\
\hline
& & & & & & & & & \\
$Co_1$ & $C(2B)=(2^2 \times G_2(4)):2$ &$G_2(4)$ & $G_2$ & $D_8$& $D_8 \times G_2(2)$ & $D_8 \times G_2(2)$ & $G_2(2)$&$G_2$& \cite{sa99}\\
         & $C(2C)=2^{11}:M_{12}.2$ & & contr. & & & & &&\\
& & & & & & & & & \\
\hline
& & & & & & & & & \\
$BM$ & $C(2A)=2\;^.\;^2E_6(2):2$ &$^2E_6(2)$& $F_4$ & $2^2$&$2^2 \times F_4(2)$&$(2^2 \times F_4(2)).2$ &$F_4(2)$& $F_4$&\cite{y05a}\\
     & $C(2C)=(2^2 \times F_4(2)):2$ &$F_4(2)$& $F_4$ & $D_8$&$D_8 \times \;^2F_4(2)$&$D_8 \times \;^2F_4(2)$ &$^2F_4(2)$ &$I_2$&\cite{segev91}\\
     & $C(2D)=2^9.2^{16} O^+_8(2).2$ &&contr. &&&&&&\\
     & & & & & & & & & \\
\hline
\hline
& & & & & & & & & \\
$Fi'_{24}$ & $C(2A)=2^. Fi_{22}:2$ &$Fi_{22}$& $\Dd_2(Fi_{22})$ &$D_8$&$D_8 \times Sp_6(2)$&$D_8 \times Sp_6(2)$&$Sp_6(2)$&$C_3$&\cite{ky}\\
           & & & & $2^2$ & $2^2 \times O_8^+(2).3$ & $(A_4 \times O_8^+(2).3).2$ & $O_8^+(2)$ &$D_4$ & \\
& & & & & & & & & \\
\hline
& & & & & & & & & \\
$HN$ & $C(2A)=2^.HS:2$ &$HS$& $\Dd_2(HS)$ &$SD_{16}$&$SD_{16} \times 5:4$&$SD_{16} \times 5:4$ &$Sz(2)$&$A_1$& \cite{y02} \\
&&&&$2^2$&$2^2 \times A_8$&$(A_4 \times A_8).2$&$L_4(2)$&$A_3$&\cite{nortwils86}\\
&&&&&&&&&\\
\hline
& & & & & & & & & \\
$M$ & $C(2A)=2 ^. BM$ &$BM$& $\Dd_2(BM)$ &$SD_{16}$&$SD_{16} \times \;^2F_4(2)$&$SD_{16} \times \;^2F_4(2)$&$ \;^2F_4(2)$&$I_2$&\cite{y05a}\\
& & & & & & & & & \\
\hline
\hline
\end{tabular}}

\vspace*{.4cm}
{\it Table 6.1: Information for $\tL_G(\Dd_2(G); \mathbb{F}_2)$}
\end{center}

\newpage
\begin{center}
{\scriptsize
\begin{tabular}{|c|l|c|l|c|l|l|c|l|}
\hline
1&2&3 &4 &5 &6 &7 &8 &9  \\
\hline
& & & & & & & & \\
$G$ & $C_G(T)$  & $H$ & $\Dd_3(G)^T$ & $V$ & $VC_G(V)$ & $N_G(V)$ & $\Dd_3(G)^V$ & Ref. \\
& & & & & & & & \\
\hline
\hline
& & & & & & & & \\
$M_{12}$ & $C(3B)=3 \times A_4$ & $L_2(3)$ & $A_1$ & $3$ & $3 \times A_4$ &$S_3 \times A_4$ & $A_1$&\cite{mgo5}\\
& & & & & & & & \\
\hline
& & & & & & & & \\
$Co_3$ & $C(3B)=3^5:(2 \times A_5)$ & &contr.&&&&&\cite{mgo4}\\
& $C(3C)=3 \times L_2(8):3$ &$R(3)$ &$A_1$&$3$&$3 \times R(3)$ &$S_3 \times R(3)$ &$A_1$ &\\
& & & & & & & & \\
\hline
& & & & & & & & \\
$J_3$ & $C(3B)=3 \times A_6$ &$L_2(9)$ & $A_1$ & $3$ & $3 \times A_6$ & $(3 \times A_6).2$ & $A_1$&\cite{mgo5}\\
& & & & & & & & \\
\hline
& & & & & & & & \\
$Co_2$ & $C(3B)=3 \times Sp_4(3):2$ &$Sp_4(3)$& $C_2$ & $3$&$3 \times Sp_4(3):2$ &$S_3 \times Sp_4(3):2$ & $C_2$&\\
& & & & & & & & \\
\hline
& & & & & & & & \\
$Fi_{22}$ & $C(3A)=3 \times U_4(3):2$ &$U_4(3)$& $C_2$ & $3$&$3 \times U_4(3):2$& $S_3 \times U_4(3):2$ & $C_2$& \cite{ky}\\
& $C(3C)= 3^5:2(A_4 \times A_4).2$& & contr.& & & & &\cite{wi84}\\
& $C(3D)=3^3.3.3^2:2A_4$& &contr. & && & &\\
& & & & & & & & \\
\hline
& & & & & & & & \\
$Fi_{23}$ & $C(3A)= 3 \times O_7(3)$& $O_7(3)$& $B_3$ & $3$ &$3 \times O_7(3)$ & $S_3 \times O_7(3)$ & $B_3$& \cite{anb99}\\
& $C(3C)= 3^6:(2 \times O_5(3))$ && contr.&&&&& \cite{wi87} \\
& $C(3D)=[3^9].(2 \times 2A_4)$ && contr.&&&&& \\
& & & & & & & & \\
\hline
& & & & & & & & \\
$Fi'_{24}$ & $C(3A)=3 \times O^+_8(3):3 $ &$O^+_8(3)$& $D_4$ &$3^2$&$3^2 \times G_2(3)$&$(3^2:2 \times G_2(3)).2$&$G_2$&\cite{ky}\\
& $C(3C)=3^7 .2. U_4(3)$ &&contr.&&&&&\cite{wi87} \\
& $C(3D)=3^2.3^4.3^6.(A_4 \times 2A_4)$&&contr.&&&&&\\
& $C(3E)=3^2 \times G_2(3)$ &$G_2(3)$& $G_2$ & $3^2$&$3^2 \times G_2(3)$ &$(3^2: 2 \times G_2(3)).2$& $G_2$&\\
& & & & & & & & \\
\hline
& & & & & & & & \\
$Th$ & $C(3A)=3 \times G_2(3)$&$G_2(3)$& $G_2$ & $3$& $3 \times G_2(3)$& $(3 \times G_2(3)):2$ &$A_1$ & \cite{wil88b}\\
& $C(3C)=3 \times 3^4:2A_6$&& contr.& & &&&\\
& & & & & & & & \\
\hline
\hline
& & & & & & & & \\
$HN$ & $C(3A)=3 \times A_9$&$A_9$&$\Dd_3(A_9)$&$3^2$&$3^2 \times A_6$&$(3^2:4 \times A_6)2^2$ &$A_1$&\cite{y05b}\\
& & & & & & & & \\
\hline
& & & & & & & & \\

$BM$ & $C(3A)=3 \times Fi_{22}:2$&$Fi_{22}$& $\Dd_3(Fi_{22})$ & $3^2$& $3^2 \times U_4(3):2^2$&$(3^2:D_8 \times U_4(3):2^2).2$ & $C_2$&\cite{y05b}\\
& & & & & & & & \\
\hline
& & & & & & & & \\
$M$ & $C(3A)=3 ^. Fi'_{24}$ & $Fi'_{24}$& $\Dd_3(Fi'_{24})$& $3^{1+2}$& $3^{1+2} \times G_2(3)$&$(3^{1+2}:2^2 \times G_2(3)).2$ & $G_2$&\cite{anw10}\\
& $C(3C)=3 \times Th$ & $Th$& $\Dd_3(Th)$ & $3^2$& $3^2 \times G_2(3)$&$(3^{1+2}:2 \times G_2(3)).2$ & $G_2$&\cite{wil88a}\\
& & & & & & & & \\
\hline
\hline
\end{tabular}}

\vspace*{.4cm}
{\it Table 6.2: Information for $\tL_G(\Dd_3(G); \mathbb{F}_3)$}
\end{center}

\newpage
\bibliographystyle{amsalpha}
\bibliography{refer}
\end{document}